\documentclass[12pt]{amsart}

\usepackage[margin=1.20in]{geometry}

\usepackage{amssymb,amsmath}
\usepackage{amsthm,enumitem}
\usepackage{hyperref}
\usepackage{mathrsfs}
\usepackage{textcomp}

\newtheorem{thm}{Theorem}[section]
\newtheorem{lem}[thm]{Lemma}
\newtheorem{prop}[thm]{Proposition}
\newtheorem{defn}[thm]{Definition}

\newtheorem{cor}[thm]{Corollary}
\newtheorem{rem}[thm]{Remark}

\numberwithin{equation}{section}

\def\Xint#1{\mathchoice
  {\XXint\displaystyle\textstyle{#1}}%
  {\XXint\textstyle\scriptstyle{#1}}%
  {\XXint\scriptstyle\scriptscriptstyle{#1}}%
  {\XXint\scriptscriptstyle\scriptscriptstyle{#1}}%
  \!\int}
\def\XXint#1#2#3{{\setbox0=\hbox{$#1{#2#3}{\int}$}
  \vcenter{\hbox{$#2#3$}}\kern-.5\wd0}}

\def\dashint{\Xint-}

                \newcommand{\pa}{\partial}
\newcommand{\va}{\varepsilon}           \newcommand{\ud}{\mathrm{d}}
\newcommand{\be}{\begin{equation}}      \newcommand{\ee}{\end{equation}}

\newcommand{\R}{\mathbb{R}}

\begin{document}

\begin{abstract}
We establish a global weighted $L^p$ estimate for the gradient of the solution to a divergence-form elliptic equations, where  the coefficients are in a weighted VMO space and the equations have singularities on a co-dimension two boundary.
\end{abstract}

\title[$L^p$ estimates, singular elliptic equations]{\textbf{On some divergence-form singular elliptic equations with codimension-two boundary: $L^p$-estimates}
\bigskip}
\subjclass{35J25, 35J75, 35A21}
\keywords{Singular elliptic equations, lower dimensional boundary, $L^p$ estimates, weighted Sobolev spaces}
\author{\medskip Jie Ji,  \  \
Jingang Xiong } 
\address[J. Ji \& J. Xiong]{School of Mathematical Sciences, Beijing Normal University, Beijing 100875, China}
\email{jie.ji@mail.bnu.edu.cn}
\email{jx@bnu.edu.cn}
\thanks{J. Xiong is partially supported by NSFC grants 12325104 and 12271028.}
\date{\today}
\maketitle

\section{Introduction}

Let $n\ge 2$,  $\Omega=\mathbb{R}^n\setminus \mathbb{R}^{n-2}
=\{x=(x',x'')\in \R^n: x'=(x^1,x^2)\neq (0,0), x''= (x^3,\dots,x^n)\}$, and  $\Gamma$ be the boundary of $\Omega$. Denote $\rho(x)=|x'|$ as the distance to $\Gamma$.  Let  $B_r'(x'_0) \subset \mathbb{R}^2$ be the open disc centered at $x_0'$ with radius $r$, and $Q_r(x_0)=B_r'(x^\prime_0)\times \prod^{n}_{i=3}(x^i_0-r, x^i_0+r)\subset \mathbb{R}^n$.
For brevity, we write $B_r'=B_r'(0)$ and $Q_r=Q_r(0)$. For fixed constants $\alpha \geq \frac{1}{2}$ and $\lambda >0$, we study  the singular elliptic equation 
\begin{equation} \label{equation}
-\sum_{i,j=1}^n\partial_i(\rho^{-2\alpha}(a^{ij}\partial_ju- F_i)	)+\lambda \rho^{-2\alpha}u=\sqrt{\lambda}\rho^{-2\alpha}f \quad \text{in} \; \Omega
\end{equation}
with the zero Dirichlet boundary condition
\begin{equation}\label{boundary}
 	u=0\quad  \mbox{on }\Gamma, 
 \end{equation} 
 where $\partial_i= \frac{\partial }{\partial x^{i}}$,   $\big(a^{ij}(x)\big)_{i,j=1}^n$ is symmetric and satisfies 
\begin{equation}\label{elliptic1}\begin{aligned}
\kappa |\xi|^2 \leq  \sum_{i,j=1}^n a^{ij}(x) \xi_i \xi_j & \leq \kappa^{-1}|\xi|^2 \quad \forall \, \xi= (\xi_1, \xi_2, \cdots,\xi_n) \in \mathbb{R}^n, 
\end{aligned}\end{equation}
for every $x\in \Omega$ and some constant $\kappa \in (0,1)$. Here, $F=(F_1, F_2,\cdots, F_n)$ and $f$ are given data. Since $\alpha \geq \frac{1}{2}$, the weight $\rho^{-2\alpha}$ is so singular that \eqref{equation}-\eqref{boundary} has well-defined weak solutions, see, e.g., R\'akosnik \cite{RJ} about the associated weighted Sobolev spaces and their trace operators. 

Such kind of equations stem, for instance,  from  the harmonic map systems  with prescribed singularities, see Weinstein \cite{G1}, Li-Tian \cite{LT}, Nguyen \cite{NL}, and Han-Khuri-Weinstein-Xiong \cite{HX,HX25}.   For problems involving a (less regular) boundary $\Gamma$ of general co-dimension $d$, David-Feneuil-Mayboroda \cite{DF2}  established some local $L^p$ estimates for solutions to a linear PDE with coefficients multiplied by $\mathrm{dist}(\cdot,\Gamma)^{d+1-n}$,  specifically in the case where $\Gamma$ is Ahlfors regular. Hence, their $p$ is confined in certain ranges. Also, the solvability of the regularity boundary value problem  associated lower dimensional boundaries was obtained by Dai–Feneuil–Mayboroda \cite{DF}. 
More recently, Fioravanti \cite{GF} derived Schauder estimates up to the boundary $\Gamma$ for elliptic equations with H\"older continuous coefficients multiplied by $\mathrm{dist}(\cdot,\Gamma)^{a}$, in the regime where $a + d \in (0,1)$.

In this paper, we would like to establish a weighted $L^p$ weak solutions theory for \eqref{equation}-\eqref{boundary}. We begin by recalling some standard notation related to weighted spaces. Let  $\omega $  be a non-negative Borel measure on $\mathbb{R}^n$ and $p\in[1,\infty)$. For a $\omega$-measurable set  $\mathcal{D} \subset \mathbb{R}^n$, define the space $L^p(\mathcal{D},\mathrm{d}\omega)$ as the set of functions for which the norm 
\[ 
\|f\|_{L^p(\mathcal{D},\,\mathrm{d}\omega)}=\left(\int_{\mathcal{D}}|f(x)|^p\mathrm{d}\omega\right)^{1/p}
\]
is finite. 
The integral average of $f$ over $\omega$-measurable set $E$ is denoted by $\dashint_{E} f \mathrm{d}\omega$.  We say a vector field  $F=(F_1,\dots ,F_n)$ belongs to $ L^p(\mathcal{D},\mathrm{d}\omega)^n$ if each component  $F_i\in L^p(\mathcal{D},\mathrm{d}\omega)$ for $i=1,\dots, n$,  and we define $\|F\|_{L^p(\mathcal{D},\,\mathrm{d}\omega)}= \||F|\|_{L^p(\mathcal{D},\,\mathrm{d}\omega)}$. 
Similarly, the weighted Sobolev space  $W^{1,p}(\mathcal{D},\,\mathrm{d}\omega)$ is defined via the norm
\[
\|f\|_{W^{1,p}(\mathcal{D},\,\mathrm{d}\omega)}=\|f\|_{L^p(\mathcal{D},\,\mathrm{d}\omega)}+\|D f\|_{L^p(\mathcal{D},\,\mathrm{d}\omega)},
\]
where $Df$ denotes the weak derivative of $f$. Let $\mathring{W}^{1,p}(\mathcal{D},\mathrm{d}\omega)$ be  the closure of $C_c^\infty({\mathcal{D}}\setminus \Gamma)$ with respect to the norm $\|\cdot \|_{W^{1,p}(\mathcal{D},\mathrm{d}\omega)}$.  We also define $W^{1,p}_0(\mathcal{D},\mathrm{d}\omega)$ as the closure of $C_c^\infty(\overline{\mathcal{D}}\setminus\Gamma)$ with respect to  the norm $\|\cdot \|_{W^{1,p}(\mathcal{D},\mathrm{d}\omega)}$. 
For convenience, we shall use the following explicit measures 
\begin{equation}
\label{eq:measures}
\begin{split}
&\mathrm{d}\mu:=\rho^{-2\alpha}\mathrm{d}x, \\
&\mathrm{d}\mu_\sigma:=\rho^{2\alpha-2\sigma-2}\mathrm{d}x,
\end{split}
\end{equation}
where $\sigma \in(0, \frac{1}{2})$. Throughout this paper, we set 
\[
\gamma = -2\alpha + 1 + \sigma.
\]
It follows from R\'akosnik \cite{RJ} that  functions in $\mathring{W}^{1,p}(Q_r(x_0),\rho^{\gamma p}\mathrm{d}\mu_\sigma)$
 have zero trace on $\Gamma$ if $Q_r(x_0) \cap \Gamma \neq \varnothing$. 

Next,  we need the following type of BMO condition for the coefficients $\big(a^{ij}(x)\big)_{i,j=1}^n$. 

\medskip 

\noindent {\bf Assumption $(\delta_0,R_0)$.} Let $\delta_0 $ and $ R_0$ be two positive constants. We say that the coefficients $(a^{ij})$ satisfy the Assumption $(\delta_0,R_0)$ in $\Omega$, if \begin{equation}
	\max_{i,j=1,2,\cdots,n}\dashint _{Q_r(x_0) \cap \Omega} |a^{ij}-[a^{ij}]_{r,x_0}|\mathrm{d}\mu_\sigma \leq \delta_0, \quad \forall ~ r\in(0,R_0], ~ x_0\in \mathbb{R}^n, 
\end{equation}
where $[a^{ij}]_{r,x_0}=\dashint_{Q_r(x_0)}a^{ij}(x)\,\mathrm{d}x$.

Note that if the coefficients satisfy the weighted VMO condition, i.e., 
\[
\lim_{R\to 0} \sup_{x_0\in \mathbb{R}^n, r\in(0,R]} \max_{i,j=1,2,\cdots,n}\dashint _{Q_r(x_0) \cap \Omega} |a^{ij}-[a^{ij}]_{r,x_0}|\mathrm{d}\mu_\sigma=0,  
\]
the Assumption $(\delta_0,R_0)$ can be fulfilled by the appropriately rescaled coefficients. In application, the Assumption $(\delta_0,R_0)$ is weaker than the VMO  condition. 

The main result in this paper is as follows. 
 
\begin{thm}\label{main theorem} Let  $\alpha \in  [\frac{1}{2},\infty)$, $\kappa\in(0,1)$, $p\in (1,\infty)$, $\sigma\in(0,\frac{1}{2})$ and $R_0\in(0,\infty)$ be given constants.   There exist positive constants $\delta_0, \lambda_0$ and $C$,  depending only on $n,\alpha,\kappa,p $ and $\sigma$, such that the following assertion holds:\\
Suppose that the coefficients $(a^{ij})$ satisfy (\ref{elliptic1}) and  the Assumption $(\delta_0,R_0)$ in $\Omega$. Then, for any $\lambda\geq\lambda_0R_0^{-2}$, $f\in L^p(\Omega,\rho^{\gamma p}\mathrm{d}\mu_\sigma)$,  and $F\in  L^p(\Omega,\rho^{\gamma p}\mathrm{d}\mu_\sigma)^n$,  there exists a unique weak solution $$u\in \mathring{W}^{1,p}(\Omega,\rho^{\gamma p}\mathrm{d}\mu_\sigma)$$ to (\ref{equation})-(\ref{boundary}). 
Moreover, $u$  satisfies the estimate 
 \begin{equation}\label{main es}
\| Du\|_{L^p(\Omega, \,\rho^{\gamma p}\mathrm{d}\mu_\sigma)}	+\sqrt{\lambda}\|u\|_{L^p(\Omega,\,\rho^{\gamma p}\mathrm{d}\mu_\sigma)}	\leq C\left(\|F\|_{L^p(\Omega,\,\rho^{\gamma p}\mathrm{d}\mu_\sigma)}+\|f\|_{L^p(\Omega,\,\rho^{\gamma p}\mathrm{d} \mu_\sigma)}	\right).
\end{equation}
\end{thm}

Our theorem is related to the recent work of Dong-Phan \cite{HT}, where weighted $L^p$ estimates were established for equations involving weights singular or degenerate at the standard co-dimension one boundary.  Extensive research exists on singular and degenerate problems of co-dimension one; see Dong-Jeon \cite{HS}, Dong-Kim \cite{HD4}, and the references therein. By employing polar coordinates for the first two variables, our approach closely follows the methodology of \cite{HT}, a framework that extends Krylov's technique \cite{K1} for second-order elliptic and parabolic equations in $\mathbb{R}^n$ with coefficients of vanishing mean oscillation (VMO). (See Dong-Kim \cite{HD1}, Kim-Krylov \cite{DN}, Krylov \cite{K3} for further related works.) 

A key difficulty in this adaptation arises from the additional singularity introduced by the angular derivatives. We overcome this via a barrier function argument and a detailed ODE analysis.

The rest of the paper is organized as follows. 
In Sect.\ref{se2}, we present several  preliminary results, including weighted embedding theorems tailored to our specific context and energy estimate in the whole space. The latter implies the existence and uniqueness of the weak solution  when $p=2$.
In Sect.\ref{se3}, we consider equations with constant coefficients. Here, we establish  appropriate  weighted H\"older estimates both in the interior and on the boundary. 
Finally, we present the proof of Theorem \ref{main theorem} and  derive  a local $L^p$ estimate in Sect.\ref{se4}.

In this paper, as usual, $C=C(\cdots)$ denotes a constant depending on the parameters in the parentheses. We adopt the summation convention whereby repeated indices are summed over.

\section{\texorpdfstring{Weighted Sobolev inequalities and $L^2$ theory}{Weighted Sobolev inequalities and L2 theory}}\label{se2}

First, we provide the definition of weak solutions. 

\begin{defn}\label{def-wek}	 
Let $\mathcal{D}\subset \mathbb{R}^n$, $n\ge 2$,  be an open set with Lipschitz boundary $\partial \mathcal{D}$. For any $f\in L^p(\mathcal{D},\rho^{\gamma p}\mathrm{d}\mu_\sigma)$  and $F\in  L^p(\mathcal{D},\rho^{\gamma p}\mathrm{d}\mu_\sigma)^n$, we say that $u\in \mathring{W}^{1,p}(\mathcal{D},\rho^{\gamma p}\mathrm{d}\mu_\sigma)$ is a weak solution of 
\[
-\sum_{i,j=1}^n\partial_i(\rho^{-2\alpha}(a^{ij}\partial_ju- F_i)	)+\lambda \rho^{-2\alpha}u=\sqrt{\lambda}\rho^{-2\alpha}f \quad \text{in } \mathcal{D}\setminus \Gamma
\]
and 
\[
u=0 \quad \mbox{on }\overline{\mathcal{D}}\cap \Gamma,
\]
if 
\begin{equation*}
\int_{\mathcal{D}}(a^{ij}\partial_iu \partial_j\varphi-F_i\partial_i \varphi) \mathrm{d}\mu+\lambda \int_{\mathcal{D}} u\varphi \mathrm{d}\mu=\sqrt{\lambda}\int_{\mathcal{D}}f\varphi \mathrm{d}\mu
\end{equation*} 
for all $ \varphi \in {W}^{1,p*}_0(\mathcal{D},\rho^{\gamma p*}\mathrm{d}\mu_{\sigma})$, where $p*=\frac{p}{p-1}$.
\end{defn}
If $\overline{\mathcal{D}}\cap \Gamma$ is empty,  the definition remains unchanged, and we therefore do not single out this case. 
Note that  $\mathring{W}^{1,p}(\Omega,\rho^{\gamma p}\mathrm{d}\mu_\sigma)={W}^{1,p}_0(\Omega,\rho^{\gamma p}\mathrm{d}\mu_\sigma)$. The above definition also makes sense when $\mathcal{D}=\mathbb{R}^n$.  Next, we prove Theorem \ref{main theorem} when $p=2$. 
\begin{lem}\label{whole integral} Let $\alpha\geq \frac{1}{2}$, $\lambda>0$, and let $a^{ij}$  be measurable functions defined on $\Omega$, such that (\ref{elliptic1}) is satisfied.
Then for each $ F\in L^2(\Omega,\mathrm{d}\mu)^n$ and $f \in L^2(\Omega,\mathrm{d}\mu)$, there exists a unique weak solution  $u\in \mathring{W}^{1,2}(\Omega,\mathrm{d}\mu)$ to (\ref{equation})-(\ref{boundary}). Moreover,
\begin{equation}\label{en es}
\|Du\|_{L^2(\Omega,\,\mathrm{d}\mu)}+\sqrt{\lambda}\|u\|_{L^2(\Omega,\,\mathrm{d}\mu)}\leq C\left(\| F\|_{L^2(\Omega,\,\mathrm{d}\mu)}+\|f\|_{L^2(\Omega,\,\mathrm{d}\mu)}\right),
\end{equation}
where $C=C(n,\kappa)$. 
\end{lem}
\begin{proof}
By the Lax-Milgram theorem,  (\ref{en es}) ensures both the existence and the uniqueness of the solution.
By the definition of weak solutions,  
we obtain 
\begin{equation*}
 \int_{\Omega}(a^{ij}\partial_ju \partial_iu- F_i\partial_iu)\mathrm{d}\mu+\lambda\int_{\Omega}u^2\mathrm{d}\mu=\int_{\Omega}\sqrt{\lambda}fu\mathrm{d}\mu.
 \end{equation*}
By the uniform ellipticity and applying the  Cauchy's  inequality with small $\varepsilon>0$,
\begin{equation}
\int_{\Omega}|D u|^2\mathrm{d}\mu+\lambda\int_{\Omega}|u|^2\mathrm{d}\mu\leq C \left(\frac{1}{4\varepsilon}\int_{\Omega}(|F|^2+f^2)\mathrm{d}\mu+\varepsilon \int_{\Omega}(|D u|^2+\lambda|u|^2)\mathrm{d}\mu\right),
\end{equation}
Choosing $C\varepsilon<\frac{1}{2}$, we obtain (\ref{en es}). Notice that $\rho^{2\gamma}\mathrm{d}\mu_\sigma=\mathrm{d}\mu$. Thus, (\ref{en es}) is equivalent to \eqref{main es} when $p=2$.
\end{proof}
We recall a general Hardy type inequality.
\begin{lem}[Sect. 1.3.1 (iv) of Maz'ya \cite{M}]\label{wp}For $r>0$, $p>1$, $s\ne 2$ and $Q_r(x_0) \cap \Gamma \neq \varnothing$. Then, we have
\begin{equation*}
\int_{Q_r(x_0)}|v|^p\rho^{-s}\mathrm{d}x \leq C\int_{Q_r(x_0)}|Dv|^p \rho^{p-s}\mathrm{d}x 
\end{equation*}
for any $v\in \mathring{W}^{1,p}(Q_r(x_0),\rho^{p-s}\mathrm{d}x)$, where $C=C(n,p,s)$.
\end{lem}

In the context of weighted spaces, by Corollary 2 in Sect. 2.1.7 of \cite{M},  we also have the Sobolev embedding. 
\begin{lem}\label{WE}
Let $\alpha \in [\frac{1}{2}, \infty)$, $\sigma\in(0,\frac{1}{2})$, $1\leq p^*<n $, and $p^* \leq p$ be given such that they  satisfy:
\begin{equation}\begin{aligned}\label{domain}
p &\leq \frac{np^*}{n-p^*}\\
\frac{n+2\alpha -2-2\sigma}{p^*} &\leq \frac{n+2\alpha -2-2\sigma}{p} +1 \\
-2\alpha+\sigma&+1+\frac{2\alpha-2\sigma}{p}>0.\end{aligned}\end{equation} 
Then for $v \in \mathring{W}^{1,p^*}(Q_2, \rho^{\gamma p^*}\mathrm{d}\mu_\sigma)$,
\[
\| v\|_{L^{p}(Q_2, \,\rho^{\gamma p} \mathrm{d}\mu_\sigma)} \leq C \|Dv\|_{L^{p^*}(Q_2, \,\rho^{\gamma p^*} \mathrm{d}\mu_\sigma)},
\]
where $C = C(n,\alpha,p,p^*)$. 
\end{lem}
\begin{proof} If the second inequality in \eqref{domain} holds with equality, then the lemma follows from Corollary 2 in Section 2.1.7 of \cite{M}. The general case then follows by H\"older's inequality.
\end{proof}

\section{ Equations with constant coefficients}\label{se3}

Let us start with local weak solutions of the homogeneous equation with constant coefficients
\begin{equation}\label{ho equation}
-\sum^n_{i,j=1}\partial_i(\rho^{-2\alpha}\overline{a^{ij}}\partial_j u)+\lambda  \rho^{-2\alpha}u =0 
\end{equation}
 in $Q_R(x_0)\backslash\Gamma$ with the partial boundary condition: 
\begin{equation*}
u=0 \quad \mbox{on } Q_R(x_0)\cap\Gamma, \quad \mbox{if }Q_R(x_0)\cap\Gamma\neq \varnothing,
\end{equation*}
where $(\overline{a^{ij}})$ is a positive definite matrix and satisfies (\ref{elliptic1}).

We shall use cylindrical coordinates
\begin{equation*}
\left\{
\begin{aligned}
x^1&=\rho \cos\theta,\\ x^2&=\rho \sin\theta,\quad \theta \in [0,2\pi],\,\rho\in[0,\infty), 
\end{aligned}\right.\end{equation*}
with $x''$ unchanged. Under this transform,  we have 
\[
\mathrm{d}x=\rho \,\mathrm{d}\rho \,\mathrm{d}\theta \,\mathrm{d}{x''},
\]
\[
D u=(\partial_\rho u,\frac{1}{\rho}\partial_\theta u,D_{x''} u )\cdot T,
\]
where  $D_{x''} =(\partial_{3},\cdots,\partial_{n})$, 
\begin{equation} \label{eq:T}
T = \begin{pmatrix} 
\cos \theta & \sin \theta &0\\ 
-\sin \theta  &  \cos \theta &0\\
0&0&I_{n-2}
\end{pmatrix}, 
\end{equation}  and $I_{n-2}$ is the identity matrix of dimensions $n-2$. 
Set
\[
\widetilde{D}=(\widetilde{D}_1,\widetilde{D}_2,\cdots,\widetilde{D}_n)=(\partial_\rho,\frac{1}{\rho}\partial_\theta, D_{x''}).
\]
Under the transform, the equation (\ref{ho equation}) is equivalent to
\begin{equation}\label{equi equation}
-\sum_{i,j=1}^n\frac{1}{\rho}\widetilde{D}_i(\rho^{1-2\alpha}b^{ij}\widetilde{D}_ju)+\lambda  \rho^{-2\alpha}u=0, \quad\end{equation}
where 
\[
B:=(b_{ij})_{i,j=1}^n= T\overline{A} T^\prime , \quad \overline{A}=(\overline{a^{ij}})_{i,j=1}^n
\]
and $T^\prime$ is the transpose of $T$.  Since $T$ is orthogonal, 
$B$ satisfies the same ellipticity condition as   $\overline{A}$.
 We shall use the notations that 
\begin{align*}
D_{\theta x''}^ku&=\sum_{\beta_2+\dots+\beta_n=k}\partial_\theta^{\beta_2}\partial_{3}^{\beta_3}\cdots\partial_{n}^{\beta_n}u \\ D_{ x''}^{k}u&=\sum_{\beta_3+\dots+\beta_n=k}\partial_{3}^{\beta_3}\cdots\partial_{n}^{\beta_n}u,
\end{align*}
where $k\ge0$ and $\beta_i\ge0$, $i=2,\dots,n$. 

\begin{lem}\label{lem3.1}
Let $u\in \mathring{W}^{l,2}(Q_R(x_0),\mathrm{d}\mu)$ be a  weak solution of (\ref{ho equation}) and $0<r<R \leq 1$. For any nonnegative integer $k\leq l-1$, we have
\begin{equation}\label{3.1}\begin{aligned}
&\|\widetilde{D} u\|_{L^2(Q_r(x_0), \mathrm{d}\mu)}+\sqrt{\lambda} \|u\|_{L^2(Q_r(x_0), \mathrm{d}\mu)} \leq C(R-r)^{-1} \|u\|_{L^2(Q_{R}(x_0), \mathrm{d}\mu)},\\[2mm] 
&\|\widetilde{D}D_{\theta x''}^{k}u\|_{L^2(Q_r(x_0), \mathrm{d}\mu)}\leq C(R-r)^{-(k+1)}\|u\|_{L^2(Q_{R}(x_0), \mathrm{d}\mu)},
\end{aligned}\end{equation}
where $C=C(n,\kappa).$
\end{lem}
\begin{proof}
For $0<r<R\leq1$, let $\xi\in C^\infty_0(Q_R(x_0)\backslash\Gamma)$ be a cutoff function, $\xi\equiv1$ in $Q_r(x_0)$ and 
$|D\xi| \leq \frac{C}{R-r}$.

By the definition of weak solutions, taking $u \xi^2$ as the test function, we obtain 
\begin{equation*}
\int_{Q_R(x_0)} \overline{a^{ij}}\xi^2 \partial_ju \partial_iu \mathrm{d}	\mu+\lambda \int_{Q_R(x_0)}  \xi^2 u^2 \mathrm{d}	\mu=-2\int_{Q_R(x_0)} \xi \overline{a^{ij}}u \partial_ju \partial_i\xi \mathrm{d}\mu.
\end{equation*}
By the uniform ellipticity condition \eqref{elliptic1} and the Cauchy's inequality, we obtain
\begin{equation}\label{energy estimate}
\int_{Q_r(x_0)} |Du|^2 \mathrm{d}	\mu+\lambda\int_{Q_r(x_0)}|u|^2\mathrm{d}\mu
\leq \frac{C}{|R-r|^2}\int_{Q_R(x_0)}|u|^2\mathrm{d}\mu.
\end{equation}
Since $|Du|=|\widetilde{D}u|$, we have
\begin{equation}\label{eq}
\int_{Q_r(x_0)} |\widetilde{D}u|^2 \mathrm{d}	\mu+\lambda\int_{Q_r(x_0)}|u|^2\mathrm{d}\mu
\leq \frac{C}{|R-r|^2}\int_{Q_R(x_0)}|u|^2\mathrm{d}\mu.
\end{equation}
Given that  $B$ is uniformly elliptic and smooth  with respect to $\theta$, by using the difference quotient method in the $\theta$ and $x''$ variables, we obtain 
\begin{equation*}
    \int_{Q_r(x_0)} |\widetilde{D} D_{x''} u|^2 \mathrm{d}	\mu+\lambda\int_{Q_r(x_0)}| D_{x''} u|^2\mathrm{d}\mu
\leq \frac{C}{|R-r|^2}\int_{Q_R(x_0)}| D_{x''} u|^2\mathrm{d}\mu,
\end{equation*} 
and\begin{equation*}
\int_{Q_r(x_0)} |\widetilde{D}\partial_\theta u|^2 \mathrm{d}	\mu+\lambda \int_{Q_r(x_0)} |\partial_\theta u|^2\mathrm{d}\mu
\leq \frac{C}{|R-r|^2}\int_{Q_R(x_0)}|\widetilde{D} u|^2\mathrm{d}\mu.
\end{equation*}
Since $b^{ij}$ are smooth  in $\theta$, we can iterate the argument above to obtain 
\begin{equation*}
\int_{Q_r(x_0)} |\widetilde{D} D_{x''}^k u|^2 \mathrm{d}	\mu+\lambda\int_{Q_r(x_0)}| D_{x''}^k u|^2\mathrm{d}\mu
\leq \frac{C}{|R-r|^2}\int_{Q_R(x_0)}| D_{x''}^k u|^2\mathrm{d}\mu.
\end{equation*}
Combined with (\ref{eq}), we can get (\ref{3.1}).
\end{proof}

Next, we present a scaling argument that will be used in the subsequent text. Let $x\in \Omega$, $r>0$ and consider the following scaling relations:  \begin{equation}\label{sc0}x=r\overline{x},\quad \overline{u} (\overline{x})=u(x), \quad \overline{\lambda}=r^{2}\lambda, \quad \mathrm{d}\overline{\mu}_\sigma=\rho^{2\alpha-2-2\sigma}\mathrm{d}\overline{x}.\end{equation}
Thus, we have
 $\rho(x)=r \rho(\overline{x})$, $Du(x)=r^{-1}D\overline{u}(\overline{x})$. We observe that
\begin{equation*}\begin{aligned}\label{}
&-\partial_i(\rho^{-2\alpha}(\overline{x})\overline{a^{ij}}\partial_j \overline{u}(\overline{x}))+\overline{\lambda} \rho^{-2\alpha}(\overline{x})\overline{u}(\overline{x}) \\ =&-r^{-2\alpha+2}\partial_i((\rho(x))^{-2\alpha}\overline{a^{ij}}\partial_j u(x))+r^{-2\alpha+2} \lambda\rho^{-2\alpha}(x)u(x).
\end{aligned}\end{equation*}
For $x_0\in \mathbb{R}^n$, when (\ref{ho equation}) holds in $Q_{r}(x_0)\backslash\Gamma $, we have
\begin{equation} \label{eq:rescaled}
-\partial_i(\rho^{-2\alpha}\overline{a^{ij}}\partial_j \overline{u}(\overline{x}))+\overline{\lambda} \rho^{-2\alpha}\overline{u}(\overline{x})= 0 \quad \mbox{in } \ Q_{1}(\overline{x_0})\backslash\Gamma.
\end{equation}
Recall that $\gamma=-2\alpha+1+\sigma$. We can derive the following equalities:
\begin{equation}\label{sc3}
\begin{aligned}
r^{-2\alpha+1} [\rho^{\gamma}\overline{u}]_{C^{\sigma}(Q_{1}(\overline{x_0}))}&=[\rho^{\gamma}u]_{C^\sigma(Q_{r}(x_0))},
\\ r^{-2\alpha}[\rho^{\gamma}D\overline{u}]_{C^\sigma(Q_{1}(\overline{x_0}))}&=[\rho^{\gamma}Du]_{C^\sigma(Q_{r}(x_0))}.
\end{aligned}
\end{equation}
 By \eqref{sc0}, we obtain $\mathrm{d}{\mu}_\sigma=r^{2\alpha-2-2\sigma-n}\mathrm{d}\overline{\mu}_\sigma$. Moreover, when $Q_2(\overline{x_0})\bigcap \Gamma=\varnothing$, we have $ \frac{1}{2}\rho({x_0})\leq  \rho({x})\leq\frac{3}{2}\rho({x_0})$ for $x\in Q_{r}(x_0)$. Furthermore, $C\rho^{2\alpha-2-2\sigma}(x_0)r^{n}\le \int_{Q_r({x_0})}\mathrm{d}{\mu}_\sigma\leq C\rho^{2\alpha-2-2\sigma}(x_0)r^{n}$, where $C$ depends only on $n$. Thus, we obtain

\begin{equation}\begin{aligned}\label{sc4}
&\rho^{2\alpha-2\sigma-2}(x_0)\dashint_{Q_{r}(x_0)}|\rho^{\gamma}u|^2\mathrm{d}\mu_\sigma \\ \leq & C r^{-2\alpha}\int_{Q_1(\overline{x_0})}|\rho^{\gamma}\overline{u}|^2\mathrm{d}\overline{\mu}_\sigma  \leq C\rho^{2\alpha-2\sigma-2}(x_0) \dashint_{Q_{r}(x_0)}|\rho^{\gamma}u|^2\mathrm{d}\mu_\sigma 
\end{aligned}	
\end{equation}and
\begin{equation}\begin{aligned}\label{sc5} &\rho^{2\alpha-2\sigma-2}(x_0)\dashint_{Q_{r}(x_0)}|\rho^{\gamma}Du|^2\mathrm{d}\mu_\sigma \\ \leq& C r ^{-2\alpha-2}	\int_{Q_1(\overline{x_0})}|\rho^{\gamma}D\overline{u}|^2\mathrm{d}\overline{\mu}_\sigma \leq C\rho^{2\alpha-2\sigma-2}(x_0)\dashint_{Q_{r}(x_0)}|\rho^{\gamma}Du|^2\mathrm{d}\mu_\sigma.	
\end{aligned}	
\end{equation}

\begin{prop}\label{prop3.2} Let $x_0\in \Omega$, $r>0$, and suppose that $Q_r(x_0)\bigcap\Gamma=\varnothing$. If $u\in W^{1,2}(Q_r(x_0),\mathrm{d}\mu)$ is a weak solution to (\ref{ho equation}) in $Q_r(x_0)$, then we have
\begin{equation}\begin{aligned}\label{in}
[\rho^{\gamma}u]_{C^\sigma(Q_{r/2}(x_0))}&\leq Cr^{-\sigma}\left(\dashint_{Q_{2r/3}(x_0)}|\rho^{\gamma}u|^2\mathrm{d}\mu_\sigma \right)^{1/2},\\
[\rho^{\gamma}D u]_{C^\sigma(Q_{r/2}(x_0))}&\leq Cr^{-\sigma}\left(\dashint_{Q_{2r/3}(x_0)}(|\rho^{\gamma}Du|^2+\lambda|\rho^{\gamma}u|^2)\mathrm{d}\mu_\sigma \right)^{1/2},	\end{aligned}\end{equation}
where  $C=C(n,\alpha, \kappa,\sigma)$.
\end{prop}

\begin{proof} We shall use the change of variables as in (\ref{sc0}) and obtain the equation \eqref{eq:rescaled} for $\overline{u}$. By the assumption $Q_r(x_0)\bigcap\Gamma=\varnothing$, we have $Q_1(\overline{x_0})\cap \Gamma=\varnothing$, 
\[
\frac13<\rho(\overline{x})<\frac53  \quad \mbox{in }Q_{2/3}(\overline{x_0}).
\]
Now we are in the setting of the standard uniformly elliptic equations with smooth coefficients 
\[
-\partial_i(\rho^{-2\alpha}\overline{a^{ij}}\partial_j \overline{u}(\overline{x}))+\overline{\lambda} \rho^{-2\alpha}\overline{u}(\overline{x})= 0 \quad \mbox{in } \ Q_{2/3}(\overline{x_0}),
\]
where $\overline{\lambda} >0$.  Similar to Lemma \ref{lem3.1}, we can  derive  Caccioppoli type inequalities for derivatives in all directions. That is, for any non-negative integer $l$ and $k=1,2,\cdots,n$,
\begin{equation*}\begin{aligned}
&\|{D}\partial^l_k \overline{u}\|_{L^2(Q_{1/2}(x_0))}+\sqrt{\overline{\lambda}} \|\partial^l_k \overline{u}\|_{L^2(Q_{1/2}(x_0))} \leq C\|\overline{u}\|_{L^2(Q_{2/3}(x_0))},
\end{aligned}\end{equation*}where $C$ is independent of $\overline{\lambda}$.
From the standard energy estimates and the Sobolev embedding, we can have
\begin{equation}\label{3.2-1}\begin{aligned}
\relax[\overline{u}]_{C^\sigma(Q_{1/2}(\overline{x_0}))} \leq & C\left(\int_{Q_{2/3}
(\overline{x_0})}|\overline{u}|^2\mathrm{d}x\right)^{1/2}, \\
[D \overline{u}]_{C^\sigma(Q_{1/2}(\overline{x_0}))}	\leq &C\left(\int_{Q_{2/3}(\overline{x_0})}(|D\overline{u}|^2+\overline{\lambda} |\overline{u}|^2)\mathrm{d}x\right)^{1/2},
\end{aligned}
\end{equation}
where $C=C(n,\kappa,\alpha,\sigma)>0$. Finally, \eqref{in} follows from \eqref{sc3}-\eqref{sc5}. 
\end{proof}
We denote 
\[
\Theta_r(x_0'')=[0,2\pi]\times \prod^{n}_{i=3}(x^i_0-r, x^i_0+r).
\] 
Next, we consider the case $Q_1(\overline{x_0})\cap \Gamma\neq \varnothing$. It suffices to assume $x_0=0$. 
\begin{lem}\label{prop3.3}Let 
 $u\in \mathring{W}^{1,2}(Q_1,\mathrm{d}\mu)$ be a weak solution of (\ref{ho equation}) in $Q_1\backslash\Gamma$. Then for any $x$ in $Q_{1/2}\backslash\Gamma$ and $k \in \mathbb{N}\cup\{0\}$, we have
\begin{equation}\label{prop-3.5}
\begin{aligned}
\rho^{-\alpha}(x)|D^k_{\theta x''} u(x)|\leq C  \left(\dashint_{Q_{1}}|u|^2\mathrm{d}\mu \right)^{1/2}, \\ \rho^{-\alpha+1}(x)|\partial_\rho u(x)|\leq C\left(\dashint_{Q_{1}}(|Du|^2+\lambda|u|^2)\mathrm{d}\mu \right)^{1/2},
\end{aligned}
\end{equation}
where $C=C(n,\alpha,\kappa,k).$
\end{lem}
\begin{proof}
We will use the cylindrical coordinates transformation in the following proof.
By applying the Sobolev embedding in $(\theta, x'')$, we have for any $x \in Q_{1/2}\backslash\Gamma$,
\begin{equation}\label{embedding}
|\partial_\rho u(\rho,\theta,x'')|\leq C\|\partial_\rho u(\rho,\cdot)\|_{W^{l,2}(\Theta_{1/2})}
\end{equation}
with some integer $l>\frac{n-1}{2}$. 
Then, using (\ref{embedding}) and Lemma \ref{lem3.1}, we have
\begin{equation}\begin{aligned}\label{3}
&\int_0^{1/2}|\partial_\rho u(\rho,\theta, x'')|^2 	\rho^{-2\alpha}\rho\mathrm{d}\rho \\ \leq &C\int_0^{1/2}	\|\partial_\rho u\|^2_{W^{l,2}(\Theta_{1/2})} \,\rho^{-2\alpha}\rho\mathrm{d}\rho \\ \leq & C\int_0^{1/2}\int_{\Theta_{1/2}}	\sum_{k\leq l}|D^{k}_{\theta x''}\partial_\rho u|^2 \rho^{-2\alpha}\rho\,\mathrm{d}\theta \,\mathrm{d}x'' \,\mathrm{d}\rho\\ \leq & C\|u\|_{L^2(Q_1,\mathrm{d}\mu)}.
\end{aligned}
\end{equation}
By (\ref{3}) and H\"older's inequality, for any $(\theta,x'')\in \Theta_{1/2}$,
\begin{equation*}\begin{aligned}
&\int_0^\rho|\partial_\rho u(\rho,\theta, x'')|\mathrm{d}\rho 
\\ \leq &C\left(\int_0^\rho|\partial_\rho u(\rho,\theta, x'')|^2\rho^{-2\alpha}\rho \,\mathrm{d}\rho\right)^{1/2}\left(\int_0^\rho \rho ^{-1+2\alpha}\mathrm{d}\rho\right)^{1/2}
\\ \leq & C \rho^\alpha\|u\|_{L^2(Q_1,\mathrm{d}\mu)}.
\end{aligned}	
\end{equation*}
By the fundamental theorem of calculus and the boundary condition $u(0,\theta,x'')=0$, we obtain
\begin{equation}\label{u}
	|u(\rho,\theta,x'')|\leq \int_0^\rho|\partial_\rho u|\mathrm{d}\rho\leq C\rho^\alpha\|u\|_{L^2(Q_1,\mathrm{d}\mu)}.
\end{equation}
Since $\partial^{\beta_i}_{i} u$ $(i=3,\cdots,n)$ satisfies the same equation as $u$ under the assumption that $\overline{a^{ij}}$ are constants, by using (\ref{u})  and Lemma \ref{lem3.1}, for any  non-negative integer $k$, we get
\begin{equation}\label{thetaz}
	|D_{x''}^{k} u(\rho,\theta,x'')|\leq C\rho^\alpha\|u\|_{L^2(Q_1,\mathrm{d}\mu)}.
\end{equation}
Similarly, for any non-negative integer $k_2$, we obtain
\begin{equation}\label{th}\begin{aligned}
|\partial^{k_2}_\theta u(\rho,\theta,x'')|	& \leq \int_0^\rho|\partial_\rho \partial^{k_2}_\theta u|\mathrm{d}\rho \\ &\leq \left(\int_0^\rho|\partial_\rho \partial^{k_2}_\theta u|^2\rho^{-2\alpha+1}\mathrm{d}\rho\right)^{1/2}\left(\int_0^\rho\rho^{2\alpha-1}\mathrm{d}\rho\right)^{1/2} \\ &\leq C\rho^{\alpha}\left(\int_0^1\sum_{k\leq l}\int_{\Theta_1}| D_{\theta x''}^{k} \partial_\rho \partial_\theta^{k_2} u|^2\mathrm{d}\theta \,\mathrm{d}x'' \,\rho^{-2\alpha+1}\mathrm{d}\rho\right)^{1/2}\\ &\leq C\rho^{\alpha}\|u\|_{L^2(Q_1,\mathrm{d}\mu)}.
\end{aligned}\end{equation}
Thus, the first inequality in \eqref{prop-3.5} follows from \eqref{thetaz} and \eqref{th}.

Next, let $\mathcal{U}=\sum^{n}\limits_{j=1}b^{1j}\widetilde{D}_ju$.
By equation (\ref{equi equation}), we can get
\begin{equation}\label{t}\begin{aligned}
|\partial_\rho (\rho^{1-2\alpha}\mathcal{U})|	&=|\rho^{-2\alpha}\partial_\theta(b^{2j}\widetilde{D}_j u)+\sum_{i=3}^{n}\rho^{1-2\alpha}b^{ij}\widetilde{D}_jD_iu+\lambda \rho^{1-2\alpha}u|
\\&\leq C\rho^{-2\alpha}(|\widetilde{D}\partial_\theta u| +|\widetilde{D} u|+\rho|\widetilde{D}D_{x''} u|+\lambda \rho|u|).
\end{aligned}
\end{equation}
By using the Sobolev inequality in $(\theta, x'')$, for any $x\in Q_{1/2}$, we see that for some integer $l>\frac{n-1}{2}$
\begin{equation}\begin{aligned}\label{SE2}
 & \sup_{\Theta_{1/2}}\left(|\widetilde{D}\partial_\theta u|+|\widetilde{D} u|+|\widetilde{D}D_{x''} u|+\lambda|u|\right)\\ \leq & C\left(\|\widetilde{D} \partial_\theta u\|_{W^{l,2}(\Theta_{1/2})}+\|\widetilde{D}u\|_{W^{l,2}(\Theta_{1/2})}+\|\widetilde{D}D_{x''} u\|_{W^{l,2}(\Theta_{1/2})}+\lambda\|u\|_{W^{l,2}(\Theta_{1/2})}\right).	
\end{aligned}\end{equation}
Thus, \eqref{t} and \eqref{SE2} give
\begin{equation}\begin{aligned}\label{differ}
&|\rho^{1-2\alpha}\mathcal{U}(\rho, \theta,x'')-(\frac{1}{2})^{1-2\alpha}\mathcal{U}(\frac{1}{2}, \theta,x'')|
\\ \leq &\int_{1/2}^\rho |\partial_\rho(\rho^{1-2\alpha}\mathcal{U})|\mathrm{d}\rho
\\ \leq &\left(\int_{1/2}^\rho |\partial_\rho(\rho^{1-2\alpha}\mathcal{U})|^2\rho^{1+2\alpha}\mathrm{d}\rho\right)^{1/2}\left(\int_{\frac{1}{2}}^\rho \rho^{-2\alpha-1}\mathrm{d}\rho\right)^{1/2}
\\ \leq &C\left(1+\rho^{-\alpha}\right)\left( \int_{1/2}^\rho (| \widetilde{D}\partial_\theta u|^2+| \widetilde{D}u|^2+\rho^2|\widetilde{D}\partial_{x''}u|^2+\lambda \rho^2|u|^2)\rho^{-2\alpha}\rho \mathrm{d}\rho\right)^{1/2}
\\ \leq &C\left(1+\rho^{-\alpha}\right)\left( \int_{1/2}^\rho (\|\partial_\theta \widetilde{D} u\|^2_{W^{l,2}(\Theta_{1/2})}+\|\widetilde{D}D_{x''}u\|^2_{W^{l,2}(\Theta_{1/2})}+\lambda\|u\|^2_{W^{l,2}(\Theta_{1/2})})\rho^{-2\alpha+1} \mathrm{d}\rho \right)^{1/2}
\\ \leq &C\left(1+\rho^{-\alpha}\right)\left(\|Du\|_{L^2(Q_{1},\mathrm{d}\mu)}+\sqrt{\lambda}\|u\|_{L^2(Q_{1},\mathrm{d}\mu)}\right),
\end{aligned}
\end{equation}where we used Lemma \ref{lem3.1} in the last inequality.
This together with the interior gradient estimate (\ref{in}) give
\begin{equation*}
	|\rho^{1-2\alpha}\mathcal{U}(\rho ,\theta,x'')|\leq C\left(1+\rho^{-\alpha}\right)\left(\|Du\|_{L^2(Q_{1},\mathrm{d}\mu)}+\sqrt{\lambda}\|u\|_{L^2(Q_{1},\mathrm{d}\mu)}\right).
\end{equation*}
Therefore, by using the definition of $\mathcal{U}$, the ellipticity condition and (\ref{differ}), we have
\begin{equation}\label{rho}
	|\partial_\rho u(\rho ,\theta,x'')|\leq C\left(\rho^{2\alpha-1}+\rho^{\alpha-1}\right)\left(\|Du\|_{L^2(Q_{1},\mathrm{d}\mu)}+\sqrt{\lambda}\|u\|_{L^2(Q_{1},\mathrm{d}\mu)}\right),
\end{equation}
which together with (\ref{thetaz}) and (\ref{th}) give
\begin{equation}\label{f}
\rho^{1-\alpha}|\widetilde{D}u(\rho,\theta,x'')|\leq C\left(\|Du\|_{L^2(Q_{1},\mathrm{d}\mu)}+\sqrt{\lambda}\|u\|_{L^2(Q_{1},\mathrm{d}\mu)}\right).
\end{equation}
Thus, we conclude that (\ref{prop-3.5}) is valid.
\end{proof}

While the proof above adapts techniques from \cite{HT}, this approach fails to produce a sharp weight estimate on the left-hand side in our setting. To address this, we develop a new method to establish the sharp bound.
Inspired by Lemma 4.4 in \cite{LT}, we first establish an almost sharp version. 

\begin{lem}\label{lem3.5}
	For any $\varepsilon>0$, there exists a constant $C_\varepsilon$ such that if $u\in \mathring{W}^{1,2}(Q_1,\mathrm{d}\mu)$ is a solution of (\ref{ho equation}), then we have
	\begin{equation*}	\sup_{Q_{1/2}\backslash\Gamma}\rho^{-2\alpha+\varepsilon}|u|\leq C_\varepsilon \left(\dashint_{Q_1}|u|^2\rho^{-2\alpha}\mathrm{d}x\right)^{1/2},
	\end{equation*}where $C_\varepsilon=C_\varepsilon(n,\alpha, \kappa,\sigma,\varepsilon)$.
\end{lem}
\begin{proof}
We represent  $\overline{A}$ in block form as \[
\overline{A}=\begin{pmatrix} 
A_2  &J\\ 
K  &A_{n-2}
\end{pmatrix},
\]
where $A_2$ and $A_{n-2}$ are the $2$-order and $(n-2)$-order submatrices respectively,  $J$ is a $2\times(n-2)$ and $K$ is an $(n-2)\times 2$ submatrix.  

By the symmetry and uniform ellipticity of $\overline{A}$, we know that  there exists an invertible matrix $E_2=[e^{ij}]_{2\times 2}$ such that $E_2'A_2E_2=I_2$, where $E_2'$ is the transpose of $E_2$ and $I_2$ is the  second-order identity matrix. Extend $E_2$ to an matrix $E=[e^{ij}]_{n\times n}$ by setting 
\[
E=\begin{pmatrix} 
E_2  &\mathbf{0}\\ 
\mathbf{0}  &I_{n-2}
\end{pmatrix},
\]
where  $\mathbf{0}$ denotes appropriately sized zero matrices.
Define $y=(y^1,\cdots,y^n)=E \cdot x$ and then $\nabla_x=E \cdot \nabla_y$.   
We can see that $$\mathrm{div}_x(\rho^{-2\alpha}\overline{A} \nabla_x u)=\mathrm{div}_y(|E_2^{-1}y'|^{-2\alpha}E'\overline{A}E \nabla_y u)=\mathrm{div}_y(|y'|^{-2\alpha}|E_2^{-1}\frac{y'}{|y'|}|^{-2\alpha}E'\overline{A}E \nabla_{y} u),$$where $y'=(y^1,y^2)$.  
 It's straightforward to observe  that $|E_2^{-1}\frac{y'}{|y'|}|^{-2\alpha}$  depends only on $\theta$ under the cylindrical coordinates and $\rho(y)=|y'|$. Let  $H=(h^{ij})_{n\times n}=|E_2^{-1}\frac{y'}{|y'|}|^{-2\alpha}E'\overline{A}E$, and $b(\theta)=|E_2^{-1}\frac{y'}{|y'|}|^{-2\alpha}$, which is a positive scalar function. 
 Since the matrix $E$ is invertible, we only need to consider the equation
 \begin{equation}\label{eq3}
 	\mathrm{div}_{y}(\rho^{-2\alpha}H \nabla_{y} u)-\lambda  \rho^{-2\alpha}u=0.
 \end{equation}
 
 Let $\beta<2\alpha$ and $0<\delta<1$. For $y_0\in Q_1$, define the function $m(y)=\rho^{\beta}+\rho^{\beta-\delta}|y-y_0|^2$. It is straightforward to demonstrate that
\begin{equation*}\begin{aligned}
	&\mathrm{div}_y\,(\rho^{-2\alpha }H \nabla_y m)\\=
	& - 2\alpha \rho^{-2\alpha-1} h^{ij} \partial_i \rho \Big[ \beta \rho^{\beta-1} \partial_j\rho+ (\beta-\delta)\rho^{\beta-\delta-1} \partial_j\rho |y- y_0|^2 + 2\rho^{\beta-\delta}(y^j - y_{0}^j) \Big] \\
& + \rho^{-2\alpha} h^{ij}\bigg[ \beta(\beta-1)\rho^{\beta-2} \partial_i \rho\,\partial_j \rho + \beta \rho^{\beta-1} \partial_{ij}\rho \\
& + (\beta-\delta)(\beta-\delta-1)\rho^{\beta-\delta-2}\partial_i \rho\,\partial_j \rho |y - y_0|^2 + (\beta-\delta)\rho^{\beta-\delta-1} \partial_{ij}\rho |y - y_0|^2 \\
&+ 2(\beta-\delta)\rho^{\beta-\delta-1} \partial_j \rho \,(y^i - y^i_{0}) + 2(\beta-\delta)\rho^{\beta-\delta-1} \partial_i \rho\, (y^j - y^j_{0}) + 2\rho^{\beta-\delta} \delta_{ij}\bigg].
\end{aligned}
\end{equation*}
Note that $h^{ij}=b(\theta)\delta^{ij}$ for $i,j=1,2$,  we have $h^{ij}\partial_i \rho\partial_j \rho=b(\theta)>0$. Consequently, there exists a real number $r_1<1$ and a positive constant $C$  such that
\begin{equation*}\begin{aligned}
	\mathrm{div}_y\,(\rho^{-2\alpha}H \nabla_y m)&\leq  C\beta(\beta-2\alpha)b(\theta) \rho^{\beta-2\alpha-2}<0  \quad \mbox{in}\, Q_{r_1} .          
\end{aligned}\end{equation*}
If $\alpha \ge 2$, we have $\beta-2\alpha-2<-\alpha$ when $\beta<2\alpha$. By Lemma \ref{prop3.3}, we have                                                                       \begin{equation*}\begin{aligned}
	\mathrm{div}_y\,(\rho^{-2\alpha}H \nabla_y m)&\leq -C \lambda \rho^{-2\alpha}|u|  \quad \mbox{in}\, Q_{r_1} .          
\end{aligned}\end{equation*}                                                                  
Take another $r_2\in (0,r_1)$. Set $c=r_1-r_2>0$. For any $\varepsilon>0$, we fix $\delta$ such that $l\delta=\alpha-\varepsilon$ for some integer $l$. We first take $m$ with $\beta=\alpha+\delta$. By Lemma \ref{prop3.3}, we see that $\rho^{-\alpha}u$ is bounded in $Q_{r_1}$. For $y_0\in Q_{r_2}$ and $y \in \partial Q_{c}(y_0) $, we have
\begin{equation*}
|u(y)|\leq \|\rho^{-\beta+\delta}u\|_{L^{\infty}(Q_{r_1})}\rho^{\beta-\delta}(y)\leq c^{-2}m\|\rho^{-\alpha}u\|_{L^{\infty}(Q_{r_1})}.
\end{equation*}
Notice that $m=0$ on $\Gamma$, and (\ref{ho equation}) holds.
By the comparison principle, we obtain
\begin{equation*}
|u(y)|	\leq c^{-2}m\|\rho^{-\alpha}u\|_{L^{\infty}(Q_{r_1})}  \quad \mbox{in } Q_{c}(y_0).
\end{equation*}
In particular, for $y = y_0$, we have 
\[|u(y_0)|	\leq c^{-2}\rho^{\alpha+\delta}(y_0)\|\rho^{-\alpha}u\|_{L^{\infty}(Q_{r_1})}\leq C\rho^{\alpha+\delta}(y_0).\] Since $y_0$ is arbitrary in $Q_{r_2}$, we have $|u(y)|\leq C\rho^{\alpha+\delta}(y)$ in $Q_{r_2}$.

Next, we take $m$ with $\beta=\alpha+2\delta$. Similarly, we can obtain $|u|\leq C\rho^{\alpha+2\delta}$ in $Q_{r_3}$ with some $ r_3<r_2$.
 Then we repeat this argument successively with $\beta=\alpha+j\delta$ for $3\leq j\leq l$. We have $|u|\leq C\rho^{2\alpha-\varepsilon}$ in some region $Q_{r_{l+1}}$.

If $\frac{1}{2} \le \alpha<2$, we take $\beta<\alpha+2$. By replacing the earlier condition $\beta < 2\alpha$ in the above argument with this choice, we find the estimate $$|u|\leq C\rho^{\alpha+2-\varepsilon} $$  holds in some region $Q_{r}$.
To derive the desired result, we now repeat the same procedure under the assumption  $\beta<2\alpha$. Note that if $\beta<2\alpha$ and $0<\varepsilon<1$, then $\beta-2\alpha-2<-2\alpha+\alpha+2-\varepsilon$. Using the previously established bound $|u|\leq C\rho^{\alpha+2-\varepsilon}$, we conclude that there exists a real number $d_1 < 1 $ such that 
\begin{equation*}\begin{aligned}
	\mathrm{div}_y\,(\rho^{-2\alpha}H \nabla_y m)&\leq -C \lambda \rho^{-2\alpha}|u|  \quad \mbox{in}\, Q_{d_1} .          
\end{aligned}\end{equation*}  
By iterating this process, we further obtain $|u|\leq C\rho^{2\alpha-\varepsilon}$ in a smaller region $Q_{d_{l+1}}$.

 If the radius $r_{l+1}$ or $d_{l+1}$ is smaller than $\frac{1}{2}$, we can also apply the interior estimate to extend validity of this estimate to extend this estimate to the region $Q_{1/2}$.
 Noticing that $\|\rho^{-\alpha}u\|_{L^\infty(Q_{1/2})}\leq C(\dashint_{Q_1}|u|^2\rho^{-2\alpha}\mathrm{d}x)^{1/2}$, we can conclude the proof.
\end{proof}

\begin{cor}\label{cor} 
Suppose that $u\in \mathring{W}^{1,2}(Q_1,\mathrm{d}\mu)$ is a weak solution of (\ref{ho equation}).
Then for any small $\varepsilon>0$ and any integer $k\ge 0$, we have	
\begin{equation}\label{cor3.5}
		\sup_{Q_{1/2}\backslash\Gamma}\rho^{-2\alpha+\varepsilon}| D^{k}_{\theta x''}u|+\sup_{Q_{1/2}\backslash\Gamma}\rho^{-2\alpha+1+\varepsilon}|\partial_\rho D^{k}_{\theta x''}u|\leq C_\varepsilon \left(\dashint_{Q_1}|u|^2\mathrm{d}\mu \right)^{1/2},
	\end{equation}
	where  $C_\varepsilon=C_\varepsilon(n,\alpha,\kappa,k,\varepsilon)$.
\end{cor}
\begin{proof} 
 For any $x_0 \in Q_{1/2}\backslash\Gamma$, denote $r=\frac{1}{2}\rho(x_0)$.  
By Proposition \ref{prop3.2} and Lemma \ref{lem3.1}, we can see that 
\begin{equation*}\begin{aligned}
|\rho^{\gamma}D u(x_0)| &\leq C\left(\dashint_{Q_{2r/3(x_0)}}(|\rho^{\gamma}Du|^2+\lambda |\rho^{\gamma}u|^2)\mathrm{d}\mu_\sigma \right)^{1/2}\\&\leq Cr^{-1}\left(\dashint_{Q_{4r/3(x_0)}} |\rho^{\gamma}u|^2\mathrm{d}\mu_\sigma \right)^{1/2}. 
\end{aligned}\end{equation*}
Given that $\rho(x)\leq \frac{5}{6}$ in $Q_{4r/3}(x_0)$, Lemma \ref{lem3.5} implies that $|u|\leq C_{\varepsilon} \rho^{2\alpha-\varepsilon}\left(\dashint_{Q_{2}}|u|^2\rho^{-2\alpha}\mathrm{d}x\right)^{1/2}$  holds in $Q_{4r/3}(x_0)$ for any $\varepsilon>0$. Thus, \begin{equation*}\begin{aligned}
|\rho^{\gamma}D u(x_0)| &\leq C_\varepsilon r^{-1}\left(\int_{Q_{4r/3(x_0)}}\mathrm{d}\mu_\sigma\right)^{-1/2}\left(\int_{Q_{4r/3(x_0)}} \rho^{2\alpha-2\varepsilon}\mathrm{d}x\right)^{1/2}\left(\dashint_{Q_2}|u|^2\rho^{-2\alpha}\mathrm{d}x\right)^{1/2} \\& \leq C_{\varepsilon}r^{-1}\rho^{\alpha+\gamma}(x_0)\rho^{\alpha-\varepsilon}(x_0)\left(\dashint_{Q_{2}}|u|^2\rho^{-2\alpha}\mathrm{d}x\right)^{1/2}\\& = C_{\varepsilon}\rho^{2\alpha-1+\gamma-\varepsilon}(x_0)\left(\dashint_{Q_{2}}|u|^2\rho^{-2\alpha}\mathrm{d}x\right)^{1/2}, 
\end{aligned}\end{equation*}where $C_\varepsilon$ is independent of $x_0$. That is \begin{equation}\label{gr}|\partial_\rho u(x)|+\rho^{-1}|\partial_\theta u(x)| \leq C_{\varepsilon}\rho^{2\alpha-1-\varepsilon}(x)\left(\dashint_{Q_{2}}|u|^2\rho^{-2\alpha}\mathrm{d}x\right)^{1/2}\quad \mbox{in}\, Q_{1/2}.
\end{equation}
 Observe that $\partial_k u$ $(k=3,\cdots,n)$ satisfy the same equation and zero boundary condition  as $u$, namely: \begin{equation}\begin{aligned}
	\partial_i(\rho^{-2\alpha}\overline{A}\partial_j \partial_{k}u)&=\lambda \partial_{k}u ,\quad k=3,\cdots,n.
\end{aligned}\end{equation} Thus, since the inequality  (\ref{gr}) remains valid  for $\partial_{k}u$, we can deduce, by applying Lemma \ref{lem3.1}, that in $Q_{1/2}\backslash\Gamma$, the following estimate holds:
\begin{equation*}\rho|D D_{x''}^{k'} u(x)|+|D_{x''}^{k'} u(x)| \leq C_{\varepsilon}\rho^{2\alpha-\varepsilon}(x)\left(\dashint_{Q_{2}}|u|^2\rho^{-2\alpha}\mathrm{d}x\right)^{1/2}.
\end{equation*}
\end{proof}

Finally, we remove the $\varepsilon$ and obtain the sharp version. 

\begin{prop}\label{optimal}
Suppose that $u\in \mathring{W}^{1,2}(Q_1,\mathrm{d}\mu)$ is a weak solution of \eqref{ho equation}.  Then for any integer $k\ge 0$, we have
	\begin{equation}\label{op}
		\sup_{Q_{1/2}\backslash\Gamma}\rho^{-2\alpha}|D^k_{\theta x''}u|+\sup_{Q_{1/2}\backslash\Gamma}\rho^{1-2\alpha}|Du|\leq C \left(\dashint_{Q_1}|u|^2\rho^{-2\alpha}\mathrm{d}x\right)^{1/2},
	\end{equation}where $C=C(n,\alpha,\kappa)$.
\end{prop}
\begin{proof}
As in the proof Lemma \ref{lem3.5}, we diagonalize $A_2$ and obtain the  equation (\ref{eq3}).
	Denote $\widetilde{H}=(\widetilde{h}^{ij})=T'HT$ with $T$ as in \eqref{eq:T}.  Hence, 
	\begin{equation*}
\begin{aligned}
& b(\theta) \partial^{2}_{\rho} u + b(\theta)\dfrac{1}{\rho^{2}} \partial^{2}_{\theta} u  +\frac{1}{\rho^2}\partial_\theta b(\theta)\partial_\theta u + \dfrac{1-2\alpha}{\rho} b(\theta)\partial_\rho u \\ =& \lambda  u- \sum_{j=3}^{n}(\widetilde{h}^{1j}D_j\partial_\rho u+2\alpha \rho^{-1}\widetilde{h}^{1j}D_ju-\rho^{-1}\partial_\theta \widetilde{h}^{2j}D_ju+ \rho^{-1}\widetilde{h}^{2j}D_j\partial_{\theta}u)-\sum_{i=3}^{n}\widetilde{h}^{ij}\widetilde{D}_jD_iu.
\end{aligned}
\end{equation*}
By Corollary \ref{cor}, we have 
\begin{equation*}
 	  b(\theta)\partial^{2}_{\rho}u + b(\theta)\dfrac{1}{\rho^{2}} \partial^{2}_\theta u +b(\theta)\dfrac{1-2\alpha}{\rho}  \partial_\rho u+\frac{1}{\rho^2}\partial_\theta b(\theta)\partial_\theta u=O(\rho^{2\alpha-1-\varepsilon}).
\end{equation*}
We shall use an ODE argument. Let $t=-\ln \rho$.  Then we have 
\begin{equation*}
\pa_\rho u = \pa_t u  \frac{\ud t}{\ud \rho}=  -e^{t } \pa_t u ,  
\end{equation*}
\begin{align*}
\pa_\rho^2 u = \pa_t(-e^{t } \pa_t u) \frac{\ud t}{\ud \rho}=e^{2t} (\pa_t u +\pa_t^2 u) ,  
\end{align*}
and thus 
\[
b(\theta)\pa_t^2 u +2\alpha b(\theta)\pa_t u + \pa_{\theta}(b(\theta) \pa_{\theta}u) = O(e^{-(2\alpha +1-\va) t }),
\]
where $O(\cdot)$ is uniform for $x''\in [-1,1]^{n-2}$.

For any $\rho\in (0,1)$ and $x''\in [-1,1]^{n-2}$,
we define  $L^2([0,2\pi],b)$ as a space endowed with the inner product \begin{equation}\label{pa}\langle u,v\rangle=\int_0^{2\pi}u(\theta) v(\theta) b(\theta) \mathrm{d}\theta\end{equation} and the corresponding norm \[\|u\|_{L^2([0,2\pi],b)}=\left(\int_{0}^{2\pi} |u(\theta)|^2b(\theta)\mathrm{d}\theta\right)^{1/2}.\] 
The space is  Hilbert  since $b(\theta)$ is a positive and smooth scalar function defined on $[0,2\pi]$. We also denote $H^1([0,2\pi],b)=\{u\,|\,u\in L^2([0,2\pi],b),\partial_\theta u \in L^2([0,2\pi],b)\}$, where $\partial_\theta u$ denotes the weak derivative of $u$ in the usual sense. We define 
\[C^\infty_\theta([0,2\pi])= \{\psi\in C^\infty([0,2\pi])\,|\, \psi 
\text{ is }2\pi\text{-periodic}\},\] and \[H^1_\theta([0,2\pi])=\mbox{closure of } C^\infty_{\theta}([0,2\pi]) \mbox{ in } H^1([0,2\pi],b).\]
Notice that $b(0)=b(2\pi)$, we obtain
\begin{equation}\label{dec}\int_0^{2\pi}\partial_\theta(b\partial_\theta u)v\mathrm{d}\theta=-\int_0^{2\pi}b\partial_\theta u\partial_\theta v\mathrm{d}\theta,\quad u,v \in C^\infty_\theta([0,2\pi]).\end{equation}
By approximation, we find \eqref{dec} holds with $C^\infty_\theta([0,2\pi])$ replaced by  $H^1_\theta([0,2\pi])$.


Given that  the function $u$ is defined on the interval $[0,2\pi]$ with respect to the angular variable $\theta$, it naturally follows that if $u\in \mathring{W}^{1,2}(Q_1,\mathrm{d}\mu)$, then $u\in H^1_\theta([0,2\pi])$  for any $\rho\in (0,1)$ and $x''\in [-1,1]^{n-2}$. Let the operator
 $L:=-\frac{1}{b}\partial_\theta(b\partial_\theta)$, with domain   
 $D(L)=H^2([0,2\pi])\cap H^1_\theta([0,2\pi])$. Then $L$ maps $D(L)$, which is a subset of $L^2([0,2\pi],b)$, into $ L^2([0,2\pi],b)$. 
By \eqref{dec}, we observe that for all $u, v\in D(L)$, 
 \begin{equation*}\langle Lu, v\rangle=\langle u,Lv\rangle.\end{equation*}
Hence, $L$ is symmetric.

Define the bilinear form $B:H^1_\theta([0,2\pi])\times H^1_\theta([0,2\pi])\rightarrow \mathbb{R} $ associated with 
\[B(u,v)=\int_0^{2\pi}b\partial_\theta u \partial_\theta v\mathrm{d}\theta + \int_0^{2\pi}b u  v\mathrm{d}\theta. \]
A straightforward calculation shows that  $B$ is bounded and coercive. Hence, the Lax-Milgram theorem applicable.  Combining this with the standard regularity theory for elliptic equations, we  conclude that for any $f\in L^2([0,2\pi],b)$, there exists a unique weak solution $u\in D(L)$ satisfying $(L+1)u=f$. Consequently, $L+1$ is surjective, and thus  $R(L+1)=L^2([0,2\pi],b)$. It follows that $L$ is a maximal monotone operator. According to Proposition 7.6. in \cite{B}, we conclude that $L$ is self-adjoint.

Let  $\Sigma=\{\lambda_k\}_{k=0}^{\infty}$ denote the set of eigenvalues 
$L$. We then have \begin{equation}\label{eigenvalue}0=\lambda_0< \lambda_1\leq\lambda_2\leq\lambda_3\leq \cdots, \,\text{and}\,\lambda_k\rightarrow \infty \, \text{as}  \, k\rightarrow \infty.\end{equation}
Also there exists an orthonormal basis $\{\phi_k\}_{k=0}^\infty$ of $L^2([0,2\pi],b)$, where
$\phi_k\in H^1_\theta([0,2\pi],b)$  is an eigenfunction corresponding to $\lambda_k$ :
\begin{equation*}\begin{cases}
-\frac{1}{b}\partial_\theta(b\partial_\theta \phi_k)=\lambda_k \phi_k \\ \phi_k(0)=\phi_k(2\pi). 
\end{cases}\end{equation*}
Owing to the regularity theory, we have $\phi_k\in C^\infty_\theta([0,2\pi])$ and thus $\phi_k$ is bounded in $[0,2\pi]$.

Thus, for any fixed $x''$, we can expand 
\[
 u (t,\theta,x'')= \sum_{k=0}^{\infty} a_k(t) \phi_k(\theta),
\] where the $x''$-dependence is omitted in the coefficients for simplicity.
It follows that 
\begin{align}\label{F}
&a_k''+2\alpha a_k'-\lambda_k a_k=O(e^{-(2\alpha +1-\va) t }),
\end{align}
where $O(\cdot)$ is uniform for $x''\in  [-1,1]^{n-2}$.

Since $|u |=O(e^{-(2\alpha-\va )t})$, the decay rates for $a_k$ will be determined  by the roots 
\(-\alpha -\sqrt{\alpha^2+\lambda_k}. 
\) 
That is $u(t)=C_k e^{-(\alpha +\sqrt{\alpha^2+\lambda_k})t}+O(e^{-(2\alpha+1-\varepsilon)t})$. 
It's easy to see that \(-\alpha -\sqrt{\alpha^2+\lambda_k}\leq -2\alpha 
\) by (\ref{eigenvalue}).
Here, there exists different optimal regularity depend on whether $\alpha$ is an integer or not. However, we are not concerned with this distinction, and  generally, we can obtain $|u|\leq C\rho^{2\alpha}$ within $Q_{1/2}$. 
By following the same procedure as in Corollary \ref{cor}, we obtain (\ref{op}).
\end{proof}

Before proceeding, we present equalities derived through the scaling under the settings \eqref{sc0}-\eqref{eq:rescaled}. Given that $x_0=0$, we obtain 
\begin{equation}\begin{aligned}\label{sc}
r^{-4\alpha+2\sigma+2}\dashint_{Q_1}|\rho^{1+\sigma-2\alpha}\overline{u}|^2\mathrm{d}\overline{\mu}_\sigma &=C\dashint_{Q_{r}}|\rho^{1+\sigma-2\alpha}u|^2\mathrm{d}\mu_\sigma, \\ r ^{-4\alpha+2\sigma}	\dashint_{Q_1}|\rho^{1+\sigma-2\alpha}D\overline{u}|^2\mathrm{d}\overline{\mu}_\sigma &=C\dashint_{Q_{r}}|\rho^{1+\sigma-2\alpha}Du|^2\mathrm{d}\mu_\sigma,
\end{aligned}
\end{equation}
where $C=C(n,\alpha,\sigma)$.
\begin{cor}\label{Holder}
For $r>0$ and $\sigma\in (0,\frac{1}{2})$. Assume $u\in \mathring{W}^{1,2}(Q_{2r},\mathrm{d}\mu)$ is a weak solution of (\ref{ho equation}). Then we have
\begin{equation*}
r^{\sigma}[\rho^{\gamma}u]_{C^{\sigma}(Q_r\backslash\Gamma)}+ r^{\sigma+1}[\rho^{\gamma}Du]_{C^{\sigma}(Q_{r}\backslash\Gamma)}\leq C \left(\dashint_{Q_{2r}}|\rho^{\gamma}u|^2\mathrm{d}\mu_\sigma \right)^{1/2},
\end{equation*}
where $C=C(n,\alpha,\kappa)$.
\end{cor}

\begin{proof}
By the scaling argument (\ref{sc3}) and \eqref{sc}, we only need to prove that
\begin{equation*}
[\rho^{1+\sigma-2\alpha}u]_{C^{\sigma}(Q_1\backslash\Gamma)}+ [\rho^{1+\sigma-2\alpha}Du]_{C^{\sigma}(Q_{1}\backslash\Gamma)}\leq C \left(\dashint_{Q_{2}}|u|^2\rho^{-2\alpha}\mathrm{d}x\right)^{1/2}.
\end{equation*}
By the interpolation inequalities and Lemma \ref{optimal}, we have
\begin{equation*}\begin{aligned}
\relax[\rho^{1+\sigma-2\alpha}u]_{C^{\sigma}(Q_1\backslash\Gamma)} & \leq C\|\rho^{1+\sigma-2\alpha}u\|_{C^{0}(Q_1)}+C\|\rho^{1+\sigma-2\alpha}Du\|_{C^{0}(Q_1)}\\ &\leq C\|u\|_{L^2(Q_{2},\mathrm{d}\mu)}.
\end{aligned}\end{equation*}
Recall that $\mathcal{U}=b^{1j}\widetilde{D}_ju$ and $x_i=(\rho_i,\theta_i,x_i'')$.
By the triangle inequality, for any $\sigma \in (0,\frac{1}{2})$, we have
\begin{equation*}\begin{aligned}
\left[\rho^{1+\sigma-2\alpha}\mathcal{U}\right]_{C^{\sigma}(Q_1\backslash\Gamma)}=&\sup_{x_1\neq x_2\in Q_1\backslash\Gamma }\frac{|\rho_1^{1+\sigma-2\alpha}\mathcal{U}(x_1)-\rho_2^{1+\sigma-2\alpha}\mathcal{U}(x_2)|}{|x_1-x_2|^{\sigma}}\\ \leq &\sup_{x_1\neq x_2\in Q_1\backslash\Gamma }\frac{|\rho_1^{1+\sigma-2\alpha}\mathcal{U}(\rho_1,\theta_1, x''_1)-\rho_2^{1+\sigma-2\alpha}\mathcal{U}(\rho_2,\theta_1, x''_1)|}{|\rho(x_1)-\rho(x_2)|^{\sigma}}\\&+\sup_{x_1\neq x_2\in Q_1\backslash\Gamma }\frac{|\rho_2^{1+\sigma-2\alpha}\mathcal{U}(\rho_2,\theta_1,x_1'')-\rho_2^{1+\sigma-2\alpha}\mathcal{U}(\rho_2,\theta_2,x_2'')|}{|x_1-x_2|^{\sigma}}:=I_1+I_2.
\end{aligned}\end{equation*}
By using the Sobolev embedding inequality in $\rho \in (0,1)$, we have \[I_1\leq C\sup_{(\theta,x'')\in \Theta_1}\|\rho^{1+\sigma-2\alpha}\mathcal{U}(\cdot,\theta,x'')\|_{W^{1,2}((0,1))}.\]
Using the equation, we obtain
 $$\partial_\rho(\rho^{1-2\alpha}\mathcal{U})=-\rho^{-2\alpha}(b^{2j}\widetilde{D}_j\partial_\theta{u}+\partial_\theta b^{2j}\widetilde{D}_j{u}+\rho \sum_{i=3}^{n}b^{ij}\widetilde{D}_jD_iu-\lambda \rho u).$$
By Proposition \ref{optimal}, we obtain
$$\rho|\partial_\rho(\rho^{1-2\alpha}\mathcal{U})|\leq C\|u\|_{L^2(Q_{1},\mathrm{d}\mu)}(1+\rho^2).$$
For any $(\theta,x'')\in \Theta_1$, 
\begin{equation}\begin{aligned}
&\int_0^1\left(|\rho^{1+\sigma-2\alpha}\mathcal{U}(\rho,\theta,x'')|^2+|\partial_\rho (\rho^{1+\sigma-2\alpha}\mathcal{U})(\rho,\theta,x'')|^2\right)\mathrm{d}\rho\\ \leq & C\int_0^1\left(|\rho^{1-2\alpha}\mathcal{U}(\rho,\theta,x'')|^2+|\partial_\rho (\rho^{1-2\alpha}\mathcal{U})(\rho,\theta,x'')|^2\right)\rho^{2\sigma}\mathrm{d}\rho \\ \leq & C \int_0^1 (\rho^{2\sigma-1}+\rho^{2\sigma}+\rho^{2\sigma+1})\mathrm{d}\rho \|u\|_{L^2(Q_{1},\mathrm{d}\mu)}\\ \leq &C \|u\|_{L^2(Q_{2},\mathrm{d}\mu)},
\end{aligned}\end{equation}where we used $\sigma>0$ in the last inequality. Then, we have \(I_1\leq C \|u\|_{L^2(Q_{2},\mathrm{d}\mu)}.\)

By the Sobolev embedding theorem and Lemma \ref{optimal}, we obtain \begin{equation*}\begin{aligned}I_2 &\leq C\|\rho^{1+\sigma-2\alpha}\mathcal{U}(\rho,\cdot,\cdot)\|_{W^{l,2}(\Theta_{1})}\\ &\leq C \rho^{1+\sigma-2\alpha}\sum_{k\leq l}\sup_{\Theta_{1}}|D_{\theta x''}^{k}\mathcal{U}|\\&\leq C\rho^{\sigma}\|u\|_{L^2(Q_{2},\mathrm{d}\mu)}.
\end{aligned}
\end{equation*}
Thus, we have
\begin{equation*}
[\rho^{1+\sigma-2\alpha}\mathcal{U}]_{C^{1/2}(Q_{1}\backslash\Gamma)}\leq C\|u\|_{L^2(Q_{2},\mathrm{d}\mu)}.	
\end{equation*}
By the definition of $\mathcal{U}$ and the uniform ellipticity of $B$, we have\begin{equation*}
[\rho^{1+\sigma-2\alpha}Du]_{C^{1/2}(Q_{1}\backslash\Gamma)}\leq C\|u\|_{L^2(Q_{2},\mathrm{d}\mu)}.	
\end{equation*} 
\end{proof}

 In the rest of this section, we study weak solutions to the inhomogeneous equation
\begin{equation} \label{equationrho}
-\sum_{i,j=1}^n\partial_i(\rho^{-2\alpha}(\overline{a^{ij}}\partial_ju- F_i))+\lambda \rho^{-2\alpha}u=\sqrt{\lambda}\rho^{-2\alpha}f 
\end{equation}
in $\Omega$  with boundary condition \begin{equation}\label{boundaryrhof}u=0 \quad \mbox{on } \Gamma,\end{equation}
where $\overline{a^{ij}}$ is the same as that in \eqref{ho equation}.

\begin{lem}\label{BMO}
	Let $x_0\in\mathbb{R}^n$, $\tau\ge2$, $r>0$, and $\sigma\in(0,\frac{1}{2})$. Assume that $u\in \mathring{W}^{1,2}(Q_{\tau r}(x_0),\mathrm{d}\mu)$ is a weak solution to  (\ref{equationrho})-\eqref{boundaryrhof} with $f\in L^2(Q_{\tau r}(x_0),\mathrm{d}\mu)$ and $F\in  L^2(Q_{\tau r}(x_0),\mathrm{d}\mu)^n$, then we have

\begin{equation}\label{fu}\begin{aligned}
	 &\dashint_{Q_r(x_0)}|\rho^{\gamma}Du-(\rho^{\gamma}Du)_{Q_r(x_0)}|\mathrm{d}\mu_\sigma +\sqrt{\lambda}\dashint_{Q_r(x_0)}|\rho^{\gamma}u-(\rho^{\gamma}u)_{Q_r(x_0)}|\mathrm{d}\mu_\sigma\\ \leq & C\tau ^{n+2(\alpha-1-\sigma)^+}\left(\dashint_{Q_{\tau r}(x_0)}(|\rho^{\gamma} F|^2+|\rho^{\gamma} f|^2)\mathrm{d}\mu_\sigma\right)^{1/2}\\&+C \tau^{-\sigma}\left(\dashint_{Q_{\tau r}(x_0)}(\lambda|\rho^{\gamma}u|^2+|\rho^{\gamma}Du|^2) \mathrm{d}\mu_\sigma\right)^{1/2},
	\end{aligned}\end{equation}	
where $C=C(n,\alpha,\kappa)$ 
and $(u)_{Q}=\dashint_{Q}u\mathrm{d}\mu_\sigma$.
\end{lem}
\begin{proof} Let $v\in W^{1,2}(\Omega, \mathrm{d}\mu)$ be a weak solution of the equation
\begin{equation}\label{v2}
	-\sum^n_{i,j=1}\partial_i(\rho^{-2\alpha}(\overline{a^{ij}}\partial_jv-F _i\,\chi_{Q_{2 r}(x_0)})	)+\lambda \rho^{-2\alpha}v=\sqrt{\lambda}\rho^{-2\alpha}f\chi_{Q_{2 r}(x_0)} \quad \text{in} \; \Omega
\end{equation}
with the boundary condition $v=0$ on $\Gamma$. 
Notice that for any $\tau\ge2$,
\begin{equation}
\frac{\int_{Q_{\tau r}(x_0)} \mathrm{d}\mu_\sigma}{\int_{Q_{2 r}(x_0)} \mathrm{d}\mu_\sigma} \le 
\begin{cases}
C\tau^{n} \quad\mbox{if } Q_{\tau r}(x_0)\cap\Gamma=\varnothing \\
C\tau^{n+2(\alpha-1-\sigma)}\quad \mbox{if }x_0=0,
\end{cases}
\end{equation}where $C=C(n)$.
Then, it follows from Lemma \ref{whole integral} that
\begin{equation}\begin{aligned}\label{v1'}
\dashint_{Q_{ r}(x_0)}(|\rho^\gamma D v|^2+\lambda|\rho^{\gamma} v|^2)\mathrm{d}\mu_\sigma &\leq C\dashint_{Q_{2 r}(x_0)}(|\rho^{\gamma} F|^2+|\rho^{\gamma} f|^2)\mathrm{d}\mu_\sigma\\&\leq C\tau^{n+2(\alpha-1-\sigma)^+} \dashint_{Q_{\tau r}(x_0)}(|\rho^{\gamma} F|^2+|\rho^{\gamma} f|^2)\mathrm{d}\mu_\sigma.
\end{aligned}
\end{equation} 
Next, we let $w=u-v$ so that $w$ is a weak solution of
\begin{equation}\label{w2}
-\sum^n_{i,j=1}\partial_i(\rho^{-2\alpha}\overline{a^{ij}}\partial_j w)+\lambda \rho^{-2\alpha}w =0 \quad in \ Q_{\tau r}(x_0)\backslash\Gamma.
\end{equation}
Using  Proposition \ref{prop3.2}, Corollary \ref{Holder}, and the scaling argument, for $\tau \geq2$, we have
	\begin{equation*}\begin{aligned}
	&\dashint_{Q_r(x_0)}|\rho^{1+\sigma-2\alpha}w-(\rho^{1+\sigma-2\alpha}w)_{Q_r(x_0)}|\mathrm{d}\mu_\sigma \\ \leq & Cr^{\sigma}[\rho^{1+\sigma-2\alpha}w]_{C^{\sigma}(Q_r(x_0)\backslash\Gamma)}\\ \leq & Cr^{\sigma}[\rho^{1+\sigma-2\alpha}w]_{C^{\sigma}(Q_{\tau r/2}(x_0)\backslash\Gamma)}
		\\ \leq & Cr^{\sigma}(\tau r)^{-\sigma}\left(\dashint_{Q_{\tau r}(x_0)}|\rho^{1+\sigma-2\alpha}w|^2\mathrm{d}\mu_\sigma\right)^{1/2}
		\\ = & C\tau^{-\sigma}\left(\dashint_{Q_{\tau r}(x_0)}|\rho^{\gamma}w|^2\mathrm{d}\mu_\sigma\right)^{1/2}.
	\end{aligned}\end{equation*}
By Lemma \ref{wp}, when $Q_1$ near the boundary $\Gamma$, we have
\begin{equation*}
\int_{Q_{1}(x_0)}|w|^2\rho^{-2\alpha}\mathrm{d}	x\leq \int_{Q_{1}}|w|^2\rho^{-2-2\alpha}\mathrm{d}x\leq C\int_{Q_{1}(x_0)}|Dw|^2\rho^{-2\alpha}\mathrm{d}x,
\end{equation*}where $C=C(n)$. 
 That is
 \begin{equation}\label{eq3.42}
\dashint_{Q_{r}(x_0)}|\rho^{1+\sigma-2\alpha}w|^2\mathrm{d}\mu_\sigma \leq Cr^{2}\dashint_{Q_{r}(x_0)}|\rho^{1+\sigma-2\alpha}Dw|^2\mathrm{d}\mu_\sigma.
\end{equation}	By Corollary \ref{Holder} and \eqref{eq3.42}, we have $$r^{\sigma}[\rho^{\gamma}Dw]_{C^{\sigma}(Q_{r}\backslash\Gamma)}\leq C \left(\dashint_{Q_{2r}}|\rho^{\gamma}Dw|^2\mathrm{d}\mu_\sigma \right)^{1/2}.$$
Thus, for the estimate of $Dw$, we similarly obtain
\begin{equation}\label{mW}\begin{aligned}
	&\dashint_{Q_r(x_0)}|\rho^{1+\sigma-2\alpha}Dw-(\rho^{1+\sigma-2\alpha}Dw)_{Q_r(x_0)}|\mathrm{d}\mu_\sigma\\ \leq & r^{\sigma}[\rho^{1+\sigma-2\alpha}Dw]_{C^\sigma(Q_r(x_0)\backslash\Gamma)}\\ \leq & r^{\sigma}[\rho^{1+\sigma-2\alpha}Dw]_{C^\sigma(Q_{\tau r/2}(x_0)\backslash\Gamma)} \\ \leq & Cr^{\sigma}(\tau r)^{-\sigma}\left(|\rho^{1+\sigma-2\alpha}D{w}|^2 \right)^{1/2}_{Q_{\tau r}(x_0)}
	\\ = & C\tau ^{-\sigma}\left(|\rho^{1+\sigma-2\alpha}Dw|^2\right)^{1/2}_{Q_{\tau r}(x_0)}.
	\end{aligned}\end{equation}
Now we observe that \begin{equation*}\begin{aligned}
	&(|\rho^{\gamma}Du-(\rho^{\gamma}Du)_{Q_r(x_0)}|)_{Q_r(x_0)}\\ \leq &2(|\rho^{\gamma}Du-\rho^{\gamma}Dw|)_{Q_r(x_0)}+(|\rho^{\gamma}Dw-(\rho^{\gamma}Dw)_{Q_r}|)_{Q_r(x_0)}\\ \leq & 2(|\rho^{\gamma}Dv|^2)^{1/2}_{Q_{r}(x_0)}+(|\rho^{\gamma}Dw-(\rho^{\gamma}Dw)_{Q_r}|)_{Q_r(x_0)}.
\end{aligned}\end{equation*}
Thus, (\ref{v1'}) and (\ref{mW})  yield
\begin{equation}\label{bdu}\begin{aligned}
&(|\rho^{\gamma}Du-(\rho^{\gamma}Du)_{Q_r(x_0)}|)_{Q_r(x_0)}\\\leq & C\tau^{n/2+(\alpha-1-\sigma)^+}\left(|\rho^{\gamma} F|^2+|\rho^{\gamma} f|^2\right)^{1/2}_{Q_{\tau r}(x_0)}+C\tau^{-\sigma}\left(|\rho^{\gamma}Dw|^2\right)^{1/2}_{Q_{\tau r}(x_0)}\\ \leq & C\tau^{n/2+(\alpha-1-\sigma)^+}\left(|\rho^{\gamma} F|^2+|\rho^{\gamma} f|^2\right)^{1/2}_{Q_{\tau r}(x_0)}+C\tau^{-\sigma}\left(|\rho^{\gamma}Du|^2\right)^{1/2}_{Q_{\tau r}(x_0)}.
\end{aligned}\end{equation}	
Similarly, we have
	
	\begin{equation}\label{bu}\begin{aligned}
&\sqrt{\lambda}(|\rho^{\gamma}u-(\rho^{\gamma}u)_{Q_r(x_0)}|)_{Q_r(x_0)}\\\leq & C\tau^{n/2+(\alpha-1-\sigma)^+}\left(|\rho^{\gamma} F|^2+|\rho^{\gamma} f|^2\right)^{1/2}_{Q_{\tau r}(x_0)}+C \sqrt{\lambda}\tau^{-\sigma}\left(|\rho^{\gamma}u|^2 \right)_{Q_{\tau r}(x_0)}^{1/2}.
\end{aligned}\end{equation}
Add (\ref{bdu}) and (\ref{bu}) together to obtain (\ref{fu}).
\end{proof}

Let \(\mathbb{Q}\) be the collection of all $Q_r(x)$, $x\in \mathbb{R}^n$, $r\in(0,\infty)$. 
Denote the maximal function and sharp function of $g$ over cylinder by
\begin{align*}
\mathcal{M}g(x)&=\sup_{\substack{x\in {Q}\\ {Q}\in \mathbb{Q} }} \dashint_{{Q}}|g|\,\mathrm{d}\mu_\sigma(y), \\
{(g)}^\sharp(x)&\sup_{\substack{x\in {Q} \\ {Q}\in  \mathbb{Q}}}\dashint_{{Q}}|g-(g)_{Q}|\,\mathrm{d}\mu_\sigma(y).
\end{align*}

\begin{thm}
Let  $p>1$, $\alpha\ge\frac{1}{2}$ and  $\sigma \in (0,\frac{1}{2})$. Suppose that $u\in \mathring{W}^{1,p}(\Omega,\rho^{\gamma p}\mathrm{d}\mu_\sigma)$ is a weak solution of (\ref{equationrho})-(\ref{boundaryrhof}) with $f\in L^p(\Omega,\rho^{\gamma p}\mathrm{d}\mu_\sigma)$ and $F\in  L^p(\Omega,\rho^{\gamma p}\mathrm{d}\mu_\sigma)^n$. Then  we have
\begin{equation*}
\| Du\|_{L^p(\Omega,\, \rho^{\gamma p}\mathrm{d} \mu_\sigma)}	+\sqrt{\lambda}\|u\|_{L^p(\Omega,\,\rho^{\gamma p}\mathrm{d}\mu_\sigma)}	\leq C\left(\|F\|_{L^p(\Omega,\,\rho^{\gamma p}\mathrm{d}\mu_\sigma)}+\|f\|_{L^p(\Omega, \,\rho^{\gamma p}\mathrm{d}\mu_\sigma)}	\right),
\end{equation*}
where  $C=C(n,\alpha,\kappa, p,\sigma)$.	
\end{thm}
\begin{proof}
We only need to prove the result for the case when $p\in (2, \infty)$, since the case when $p \in (1, 2)$ can be established using a duality argument, similar to what will show in the proof of Theorem \ref{main theorem} later.
    
By Lemma \ref{BMO}, if $x\in Q_r(x_0)$, 
\begin{equation}\label{C}\begin{aligned}
&\left({|\rho^{\gamma}Du|}-({|\rho^{\gamma}Du|})_{Q_r(x_0)}\right)_{Q_r(x_0)}+\sqrt{\lambda}\left({|\rho^{\gamma}u|}-({|\rho^{\gamma}u|}\right)_{Q_r(x_0)})_{Q_r(x_0)}\\ \leq & C \tau^{-\sigma}H_1(x)+C\tau ^{n/2+(\alpha-1-\sigma)^+}H_2(x),
\end{aligned}\end{equation}
where $$H_1(x)=\sqrt{\lambda}[\mathcal{M}(|\rho^{\gamma}u|^2)(x)]^{1/2}+[\mathcal{M}(|\rho^{\gamma}Du|^2)(x)]^{1/2}$$  and $$H_2(x)=[\mathcal{M}(|\rho^{\gamma}F|^2)(x)]^{1/2}+[\mathcal{M}(|\rho^{\gamma}f|^2)(x)]^{1/2}.$$
It follows that for any $x\in \Omega$ and $Q\in \mathbb{Q}$ such that $x\in Q$,  (\ref{C}) holds when $Q_r$ is replaced by $Q$.
 Then, we maximize with respect to $Q\in \mathbb{Q}$ such that $x\in Q$, we have
\begin{equation*}
\begin{aligned}
(\rho^{\gamma}Du)^\sharp(x)+\sqrt{\lambda}(\rho^{\gamma}u)^\sharp(x) \leq  C \tau^{-\sigma}H_1(x)+C\tau^{n/2+(\alpha-1-\sigma)^+}H_2(x).
\end{aligned}	
\end{equation*}
Now, by the Fefferman-Stein Theorem (see \cite[Theorem 5]{FS}), we obtain
\[\|\rho^{\gamma}Du\|_{L^p}+\sqrt{\lambda}\|\rho^{\gamma}u\|_{L^p}\leq C \tau^{-\sigma}\|H_1\|_{L^p}+C\tau^{n/2+(\alpha-1-\sigma)^+}\|H_2\|_{L^p},\]
where $L^p=L^p(\Omega,\,d\mu_\sigma )$.
From the definition of $H_1$, $H_2$ and the Hardy-Littlewood maximal function theorem (recall that $p>2$) it follows that
\begin{equation*}
\begin{aligned}
&\|\rho^{\gamma}Du\|_{L^p}+\sqrt{\lambda}\|\rho^{\gamma}u\|_{L^p}\\ \leq &C\tau^{n/2+(\alpha-1-\sigma)^+}\left(\|\rho^{\gamma}F\|_{L^p}+\|\rho^{\gamma}f\|_{L^p}\right)+C\tau^{-\sigma}\left(\|\rho^{\gamma}Du\|_{L^p}+\sqrt{\lambda}\|\rho^{\gamma}u\|_{L^p}\right).
\end{aligned} \end{equation*}
Choosing $\tau$ large enough such that $C \tau^{-\sigma}\leq \frac{1}{2}$,  we have 
\begin{equation*}
\begin{aligned}
\|\rho^{\gamma}Du\|_{L^p}+\sqrt{\lambda}\|\rho^{\gamma}u\|_{L^p}\leq &C \left(\|\rho^{\gamma}F\|_{L^p}+\|\rho^{\gamma}f\|_{L^p}\right).
\end{aligned}
\end{equation*}
\end{proof}

\section{Equations with general coefficients}
\label{se4}

From the previous section's estimates and a minor adaptation of the argument in \cite{HT}, we shall finalize the proof of the main theorem. Moreover, we will establish a local $L^p$ estimate. 

\begin{prop}\label{4.1}
Let $\delta_0\in (0,1)$, $\alpha\ge\frac{1}{2}$, $r>0$, $\tau>4$, $x_0\in \mathbb{R}^n$, and $q>2$. Suppose that $u\in \mathring{W}^{1,2}(Q_{\tau r}(x_0),\mathrm{d}\mu)$ is a weak solution of (\ref{equation})-(\ref{boundary}) in $Q_{\tau r}(x_0)\backslash \Gamma$ with $f\in L^2(Q_{\tau r}(x_0),\mathrm{d}\mu)$ and $F\in L^2(Q_{\tau r}(x_0),\mathrm{d}\mu)^n$. If Assumption $(\delta_0,R_0)$ is satisfied and $\mathrm{spt}(u)\subset\mathbb{R}^{n-1}\times (x^j_0-R_0r_0,x^j_0+R_0r_0)$ for some $r_0>0$, some $j={3,\cdots,n}$ and $x^j_0 \in \mathbb{R}$,
 then we have 
\begin{equation}\begin{aligned}\label{BMOG}
	 &\dashint_{Q_r(x_0)}|\rho^{\gamma}Du-(\rho^{\gamma}Du)_{Q_r(x_0)}|\mathrm{d}\mu_\sigma +\sqrt{\lambda}\dashint_{Q_r(x_0)}|\rho^{\gamma}u-(\rho^{\gamma}u)_{Q_r(x_0)}|\mathrm{d}\mu_\sigma\\ \leq & C\tau^{n/2+(\alpha-1-\sigma)^+}\left(\left(\dashint_{Q_{\tau r}(x_0)}\left(|\rho^{\gamma} F|^2+|\rho^{\gamma} f|^2\right)\mathrm{d}\mu_\sigma\right)^{1/2}\right.\\&\left.+({\delta_0}^{\frac{q-2}{q}}+{r_0}^{\frac{q-2}{q}})\left(\dashint_{Q_{\tau r}(x_0)}|\rho^{\gamma}Du|^q\mathrm{d}\mu_\sigma \right)^{1/q}\right)\\& +C \tau^{-\sigma}\left(\dashint_{Q_{\tau r}(x_0)}\left(|\rho^{\gamma}Du|^2+\lambda|\rho^{\gamma}u|^2\right)\mathrm{d}\mu_\sigma \right)^{1/2},
	\end{aligned}\end{equation}
where $C=C(n,\alpha,\kappa,q, \sigma)$.
\end{prop}
\begin{proof}
Let $\widetilde{F}=(\widetilde{F}_1,\widetilde{F}_2,\cdots,\widetilde{F}_n)$, $\widetilde{F}_i=(a^{ij}-\overline{a^{ij}})\partial_j u$, where $\overline{a^{ij}}=\dashint_{Q _{10r}(x_0)}a^{ij}\mathrm{d}x$.
	If $r\in (0,\frac{R_0}{10})$, by H\"older's inequality, and Assumption $(\delta_0,R_0)$, we have
	\begin{equation*}\begin{aligned}
	&\dashint_{Q _{r}(x_0)}|\rho^{\gamma} \widetilde{F}(x)|^2\mathrm{d}\mu_\sigma\\ \leq &\left(\dashint_{Q _{ r}(x_0)}|a^{ij}-\overline{a^{ij}}|^{(q-2)/q}\mathrm{d}\mu_\sigma\right)^{\frac{q-2}{q}}\left(\dashint_{Q _{r}(x_0)}|\rho^\gamma Du|^q\mathrm{d}\mu_\sigma\right)^{2/q} \\ \leq & C \delta_0^{1-2/q}\left(\dashint_{Q_{r}(x_0)}|\rho^\gamma Du|^q\mathrm{d}\mu_\sigma\right)^{2/q}.
	\end{aligned}\end{equation*}
By the scaling argument, we obtain	
	\begin{equation*}\begin{aligned}
	\dashint_{Q _{\tau r}(x_0)}|\rho^{\gamma} \widetilde{F}(x)|^2\mathrm{d}\mu_\sigma \leq  C \delta_0^{1-2/q}\left(\dashint_{Q_{\tau r}(x_0)}|\rho^\gamma Du|^q\mathrm{d}\mu_\sigma\right)^{2/q}.
	\end{aligned}\end{equation*}
Whenever $Q_{\tau r}(x_0)$ lies in the interior of $\Omega$ or $Q_{\tau r}(x_0)$ is near the boundary $\Gamma$, we obtain the following result in both cases:
\begin{equation}\label{ib}
	\dashint_{Q_{r}(x_0)}\chi_{(x^j_0-R_0r_0,\,x^j_0+R_0r_0)}(x_j)\mathrm{d}\mu_\sigma=Cr^{-1}R_0r_0.
\end{equation}
If $r\geq \frac{R_0}{10}$, given that $spt(u)\subset \mathbb{R}^{n-1}\times (x^j_0-R_0r_0,x^j_0+R_0r_0)$  and  using (\ref{ib}) along with the boundedness of $a^{ij}$, we have 
\begin{equation*}\begin{aligned}
	&\dashint_{Q _{\tau r}(x_0)}|\rho^{\gamma} \widetilde{F}(x)|^2\mathrm{d}\mu_\sigma\\ \leq & C \left(\dashint_{Q _{\tau r}(x_0)}\chi_{(x^j_0-R_0r_0,x^j_0+R_0r_0)}(x_j)\mathrm{d}\mu_\sigma\right)^{(q-2)/q}\left(\dashint_{Q_{\tau r}(x_0)}|\rho^\gamma Du(x)|^{q}\mathrm{d}\mu_\sigma \right)^{2/q}\\ \leq & C{r_0}^{1-2/q}\left(\dashint_{Q _{\tau r}(x_0)}|\rho^\gamma Du(x)|^{q}\mathrm{d}\mu_\sigma \right)^{2/q}.
\end{aligned}
\end{equation*}
Therefore, in both case we have
\begin{equation*}
	\begin{aligned}
	\dashint_{Q _{\tau r}(x_0)}|\rho^\gamma \widetilde{F}(x)|^2\mathrm{d}\mu_\sigma  \leq  C({r_0}^{1-2/q}+\delta_0^{1-2/q})\left(\dashint_{Q _{\tau r}(x_0)}|\rho^\gamma Du(x)|^{q}\mathrm{d}\mu_\sigma\right)^{2/q}.
\end{aligned}
\end{equation*}
By Lemma \ref{BMO}, we obtain (\ref{BMOG}).
\end{proof}

\begin{proof}[Proof of Theorem \ref{main theorem}]
	
We split the proof into two cases when $p\in (2, \infty)$ and when $p \in (1, 2)$.

{\bf Case I}: $p\in (2, \infty)$.

We first consider the case of $spt(u)\subset\mathbb{R}^{n-1}\times (x_0^k-R_0r_0,x_0^k+R_0r_0)$ with some $x_0^k\in \mathbb{R}$ and $r_0\in (0
,1)$.
Let $q\in(2,p)$ be fixed. By H\"older's inequality and using $2\alpha-2-2\sigma >-2$, we have $u\in W^{1,q}_{loc}(\Omega,\,\rho^{\gamma q}\mathrm{d}\mu_\sigma)$ when $u\in W^{1,p}_{loc}(\Omega,\,\rho^{\gamma p}\mathrm{d}\mu_\sigma)$. Applying Proposition \ref{4.1}, for each $r>0$ and $x_0 \in \Omega$, we can write $$u=v+w,\quad \text{in}\ Q_{\tau r}(x_0)$$
where $v$ and $w$ satisfy (\ref{v2}) and (\ref{w2}).

By the definition of sharp function, H\"older's inequality and Proposition \ref{4.1},  we can see that
\begin{equation*}\begin{aligned}
&(\rho^\beta Du)^\sharp+\sqrt{\lambda}(\rho^\gamma u)^\sharp	\\ \leq & C\tau^{n/2+(\alpha-1-\sigma)^+}\left(\mathcal{M}(|\rho^\gamma f|^2)^{1/2}+\mathcal{M}(|\rho^\gamma F|^2)^{1/2}\right)\\&+C\tau^{-\sigma}\left(\sqrt{\lambda}\mathcal{M}(|\rho^\gamma u|^2)^{1/2}+\mathcal{M}(|\rho^\gamma Du|^2)^{1/2}\right) \\ &+C\tau^{n/2+(\alpha-1-\sigma)^+}({\delta_0}^{1/2-1/q}+{r_0}^{1/2-1/q})\mathcal{M}(|\rho^\gamma Du|^q)^{1/q} .
\end{aligned}\end{equation*}

For fixed $\alpha> \frac{1}{2}$, it can be readily verified that $(\Omega, \mu_\sigma)$ constitutes a space of homogeneous type. By the generalized Fefferman-Stein theorem on sharp functions, and the Hardy-Littlewood maximal function theorem, we obtain
\begin{equation*}\begin{aligned}
	& \|\rho^\gamma Du\|_{L^p}+\sqrt{\lambda} \|\rho^\gamma u\|_{L^p}\\ \leq & C\left( \|(\rho^\gamma Du)^\sharp\|_{L^p}+\sqrt{\lambda} \|(\rho^\gamma u)^\sharp\|_{L^p}
	\right)\\ \leq & C\tau^{n/2+(\alpha-1-\sigma)^+}\left(\|\rho^\gamma f \|_{L^p}+\|\rho^\gamma F\|_{L^p}\right)+C \tau^{-\sigma}\left(\sqrt{\lambda}\|\rho^\gamma u\|_{L^p}+\|\rho^\gamma Du\|_{L^p}\right)\\ &\left.+C\tau^{n/2+(\alpha-1-\sigma)^+}({\delta_0}^{1/2-1/q}+{r_0}^{1/2-1/q})\|\rho^\gamma Du\|_{L^p}\right.,
\end{aligned}\end{equation*}
where $L^p=L^p(\Omega,\mathrm{d}\mu_\sigma)$.  
First, select $C\tau^{-\sigma} \leq \frac{1}{2}$. Then, choose $\delta_0$ and $r_0$ to be sufficiently small so that $C\tau^{n/2+(\alpha-1-\sigma)^+}({\delta_0}^{1/2-1/q}+{r_0}^{1/2-1/q})\leq \frac{1}{2}$, 
we obtain
\begin{equation*}
\|\rho^\gamma Du\|_{L^p}+\sqrt{\lambda}\|\rho^\gamma u\|_{L^p}	\leq C(\|\rho^\gamma f\|_{L^p}+\|\rho^\gamma F\|_{L^p}).
\end{equation*}

Next, we remove the assumption that $\mathrm{spt}(u)\subset\mathbb{R}^{n-1}\times (x_0^k-R_0r_0,x_0^k+R_0r_0)$ with some $x_0^k\in \mathbb{R}$ and $r_0\in (0,1)$.  We will use a partition of unity argument to obtain the results. The details of proof are similar to those in the Proof of Theorem 2.2. in \cite{HT}.
Take non-negative cut-off function $\xi(z^k)\in C^\infty_0(x_0^k-R_0r_0,x_0^k+R_0r_0)$ satisfying 
\begin{equation}\label{xi}
\int_\mathbb{R}	\xi ^p(x^k)\mathrm{d}x^k=1,\quad \int_\mathbb{R}	|\xi ^\prime(x^k)|^p\mathrm{d}x^k\leq \frac{C}{(R_0r_0)^p}.
\end{equation}
For any $\eta \in \mathbb{R}$, let $u^{\eta}(x)=u(x)\xi(x^k-\eta)$, then $u^{\eta}$ is a weak solution of
\begin{equation}\label{ctn}
-\sum^n_{i,j=1}\partial_i(\rho^{-2\alpha}(a^{ij}\partial_j u^{\eta}-F_i^{\eta}))+\lambda  \rho^{-2\alpha}u^{\eta} =\sqrt{\lambda }\rho^{-2\alpha}f^\eta
\end{equation}
with the boundary condition $u^{\eta}(x)=0$ on $\Gamma$,
where 
$$F_i^\eta(x)=F_i(x)\xi(x^k-\eta)+a^{ij}u \partial_j\xi,\quad f^\eta(x)=f(x)\xi(x_k-\eta)-\lambda^{-1/2}(a^{ij}\partial_ju-F_i)\partial_{j}\xi.$$
Since $\mathrm{spt}(u^\eta)\subset \mathbb{R}^{n-1}\times (x_0^k-R_0r_0,x_0^k+R_0r_0)$, we can obtain the estimate 
\begin{equation*}
\|\rho^\gamma Du^\eta\|_{L^p}+\sqrt{\lambda}\|\rho^\gamma u^\eta\|_{L^p}	 \leq C \left(\|\rho^\gamma f^\eta\|_{L^p}+\|\rho^\gamma F^\eta\|_{L^p}\right),
\end{equation*}
where $L^p=L^p(\Omega,\,d\mu_\sigma )$.
Integrating with respect to \( \eta \), we get
\begin{equation*}\label{}\begin{aligned}
&\int_{\mathbb{R}}\left(\|\rho^\gamma Du^{\eta}\|_{L^{p}}^{p} + \lambda^{p/2}\|\rho^\gamma  u^{\eta}\|_{L^{p}}^{p}\right)\mathrm{d}\eta \\
\leq & C\int_{\mathbb{R}}\left(\|\rho^\gamma  F^{\eta}\|_{L^{p}}^{p} + \|\rho^\gamma  f^{\eta}\|_{L^{p}}^{p}\right) \mathrm{d}\eta.
\end{aligned}\end{equation*}
It follows from the Fubini theorem and (\ref{xi}) that

\begin{equation*}\begin{aligned}
\|\rho^\gamma u\|_{L^{p}}^{p}=\int_{\mathbb{R}}\|\rho^\gamma u^{\eta}\|_{L^{p}}^{p}\, \mathrm{d}\eta.
\end{aligned}\end{equation*}
Similarly,
\begin{equation*}\begin{aligned}
\|\rho^\gamma Du\|_{L^{p}}^{p}&=\int_{\Omega} \int_{\mathbb{R}} |\rho^\gamma Du(x)|^{p} |\xi(x^k-\eta)|^{p}\, \mathrm{d}\eta\, \mathrm{d} \mu_\sigma\\&\leq \int_{\mathbb{R}}\|\rho^\gamma Du^{\eta}\|_{L^{p}}^{p}\, \mathrm{d} \eta + \int_{\mathbb{R}}\|\rho^\gamma u D\xi\|_{L^{p}}^{p}\, \mathrm{d}\eta \\&\leq \int_{\mathbb{R}}\|\rho^\gamma Du^{\eta}\|_{L^{p}}^{p}\, \mathrm{d} \eta + \frac{C}{(R_0r_0)^p}\int_{\mathbb{R}}\|\rho^\gamma u \|_{L^{p}}^{p}\, \mathrm{d} \eta.
\end{aligned}\end{equation*}

Since \( r_{0} \) depends only on \( \alpha \), \( \kappa \), and \( p \), 
 by referring to the definitions of \( f^{\eta} \) and \( F_i^{\eta} \), along with equation (\ref{xi}), and applying the Fubini theorem, we obtain \begin{equation*}\begin{aligned}
 \int_{\mathbb{R}}\|\rho^\gamma F^{\eta}\|_{L^{p}}^{p}\, \mathrm{d}\eta & \leq \|\rho^\gamma F\|_{L^{p}}^{p}+C{R_0}^{-p} \|\rho^\gamma u\|_{L^{p}}^{p},\\
\int_{\mathbb{R}}\|\rho^\gamma f^{\eta}\|_{L^{p}}^{p}\, \mathrm{d}\eta &\leq \|\rho^\gamma f\|_{L^{p}}^{p}+C\lambda^{-p/2}{R_0}^{-p}\left(\|\rho^\gamma F\|_{L^{p}}^{p}+\|\rho^\gamma Du\|_{L^{p}}^{p}\right).
\end{aligned}\end{equation*}
All the estimates above imply
\begin{equation*}\begin{aligned}
&\|\rho^\gamma Du\|_{L^{p}}+\sqrt{\lambda}\|\rho^\gamma u\|_{L^{p}}\\ \leq &C (1+\lambda^{-1/2}{R_0}^{-1})\|\rho^\gamma F\|_{L^{p}}+C\|\rho^\gamma f\|_{L^{p}}+ C \lambda^{-1/2}{R_0}^{-1}\left(\|\rho^\gamma Du\|_{L^{p}}+\sqrt{\lambda}\|\rho^\gamma u\|_{L^{p}}\right).
\end{aligned}\end{equation*}
Now, we choose $\lambda_0$  such that $\lambda_0=(2C)^2$. For $\lambda\geq \lambda_0 {R_0}^{-2}$, we have $C\lambda^{-1/2} R_0^{-1}\leq \frac{1}{2}$. Thus
\begin{equation*}\begin{aligned}
\|\rho^\gamma Du\|_{L^{p}}+\sqrt{\lambda}\|\rho^\gamma u\|_{L^{p}} \leq C\left(\|\rho^\gamma F\|_{L^{p}}+\|\rho^\gamma f\|_{L^{p}}\right).
\end{aligned}\end{equation*}
  
{\bf Case II}: $p \in (1, 2)$. 

We employ the method of duality. Similar ideas are utilized in the proofs of Theorem 4.1 in \cite{HT2} and Theorem 4.3 in \cite{HT}.
Let $q = \frac{p}{p-1} \in (2,\infty)$ and let $G:\Omega \rightarrow \mathbb{R}^n$ and  $g: \Omega \rightarrow \mathbb{R}$ be measurable functions such that  $|G| +|g| \in L^q(\Omega, \mathrm{d} \mu_\sigma)$. We consider the adjoint problem in $\Omega$
\begin{equation} \label{adj-eqn-sim}
-\sum_{i,j=1}^{n}\partial_i\big(\rho^{-2\alpha}(a_{ji} \partial_{j} v - \tilde{G}_i)\big)+ \lambda\rho^{-2\alpha}  v=  \sqrt{\lambda} \rho^{-2\alpha} \tilde{g}
\end{equation}
in $\Omega$ with the boundary condition
\begin{equation}  \label{v-0721-eqn}
v =0 \quad \text{on} \ \Gamma,
\end{equation}
where
\[
\tilde{G}(y) = \rho^{-\gamma}G(y), \quad \tilde{g}(y) =  \rho^{-\gamma} g(y) .
\]
Observe that
\[
\begin{split}
\|\rho^{\gamma} \tilde{G}\|_{L^q} = \|G\|_{L^q} ,\quad 
\|\rho^{\gamma} \tilde{g}\|_{L^q} = \|g\|_{L^q},
\end{split}
\]where $L^q=L^q(\Omega,\,\mathrm{d}\mu_\sigma )$.
By {\bf Case I}, there is a unique solution $v \in \mathring{W}^{1,q}(\Omega,  \rho^{\gamma q}\mathrm{d}\mu_\sigma)$ to \eqref{adj-eqn-sim}-\eqref{v-0721-eqn}, which satisfies
\begin{equation}
              \label{eq10.35}
          \begin{split}
& \int_{\Omega} \big(|\rho^{\gamma}Dv|^q + \lambda^{q/2} |\rho^{\gamma} v|^q \big) \,\mathrm{d}\mu_\sigma(y) \\
\leq & C\int_{\Omega} \big(|G|^q + |g|^q \big) \,\mathrm{d}\mu_\sigma(y).
\end{split}
\end{equation}
 Then,  by testing \eqref{equation} with $v$, and testing \eqref{adj-eqn-sim} with $u$, and by using the definitions of $\tilde{G}$ and $\tilde{g}$, we obtain
\[
\begin{split}
 \int_{\Omega}\big( \rho^{-\gamma}G \cdot  D u
+  \sqrt{\lambda}  \rho^{-\gamma}g  u \big)\,\mathrm{d}\mu 
= \int_{\Omega}\big( F\cdot  D v +  \sqrt{\lambda} f v \big)\,\mathrm{d}\mu.
\end{split}
\]
It's equivalent to
\[
\begin{split}
& \int_{\Omega}\left( G \cdot (\rho^{\gamma} D u)
+  \sqrt{\lambda}  g \rho^{\gamma} u \right)\,d\mu_\sigma \\
=& \int_{\Omega}\left( (\rho^{\gamma}F)\cdot  (\rho^{\gamma}D v) +  \sqrt{\lambda} (\rho^{\gamma}f) (\rho^{\gamma}v) \right)\,d\mu_\sigma.
\end{split}
\]
Now,  it follows from H\"older's inequality and \eqref{eq10.35} that
\[
\begin{split}
& \left|\int_{\Omega}\big(G\cdot (\rho^{\gamma} D u)
+  \sqrt{\lambda} g \rho^{\gamma} u \big)\,\mathrm{d}\mu_\sigma\right| \\
\leq & C \|\rho^{\gamma} F\|_{L^p} \|\rho^{\gamma}  D v\|_{L^q} +  C\sqrt{\lambda} \|\rho^{\gamma} f\|_{L^{p}} \|\rho^{\gamma}v\|_{L^q}\\
\leq &  C\Big(\|\rho^{\gamma} F\|_{L^p} + \|\rho^{\gamma} f\|_{L^{p}}\Big)
\Big( \|G\|_{L^q} + \|g\|_{L^q} \Big).
\end{split}
\]
From the last estimate and as $G$ and $g$ are arbitrary, we obtain \begin{equation*}\begin{aligned}
\|\rho^\gamma Du\|_{L^{p}}+\sqrt{\lambda}\|\rho^\gamma u\|_{L^{p}}\leq C\left(\|\rho^\gamma F\|_{L^{p}}+\|\rho^\gamma f\|_{L^{p}}\right).
\end{aligned}\end{equation*}
\end{proof}

Finally, we present a local boundary  estimate for the equation 
\begin{equation} \label{nequation}
-\sum_{i,j=1}^n\partial_i(\rho^{-2\alpha}(a^{ij}\partial_ju- F_i))=\rho^{-2\alpha}f \quad \text{in} \; Q_2
\end{equation}
with the zero Dirichlet boundary condition
\begin{equation}\label{b0}
 	u=0\quad  \mbox{on }Q_2\cap \Gamma. 
 \end{equation} 

\begin{cor}
Let $\alpha \in [\frac{1}{2},\infty)$, $\kappa \in (0,1)$, $R_0 \in (0,\infty)$, $1\leq p^*<n $, and $p^* \leq p$  satisfy \eqref{domain}. Then there exists $\delta_0 = \delta_0(n,\kappa,\alpha,p,p^*) \in (0,1)$ such that the following assertion holds. Suppose that \eqref{elliptic1} and Assumption
$(\delta_0,R_0)$ are satisfied. If $u \in \mathring{W}^{1,p^*}(Q_2,\rho^{\gamma p^*}\mathrm{d}\mu_\sigma)$ is a weak solution of \eqref{nequation}-\eqref{b0} with $F \in L^p(Q_2,\rho^{\gamma p}\mathrm{d}\mu_\sigma)^n$ and $f \in L^{p^*}(Q_2,\rho^{\gamma p^*}\mathrm{d}\mu_\sigma)$, then $u \in \mathring{W}^{1,p}(Q_1,\rho^{\gamma p}\mathrm{d}\mu_\sigma)$ and
\begin{equation}\label{eq10}
\begin{aligned}
&\|u\|_{L^p(Q_1,\,\rho^{\gamma p}\mathrm{d}\mu_\sigma)} + \|Du\|_{L^p(Q_1,\,\rho^{\gamma p}\mathrm{d}\mu_\sigma)} \\
\leq &C\left(\|Du\|_{L^{p^*}(Q_2,\,\rho^{\gamma p^*}\mathrm{d}\mu_\sigma)}+         \|F\|_{L^p(Q_2,\,\rho^{\gamma p}\mathrm{d}\mu_\sigma)} + C\|f\|_{L^{p^*}(Q_2,\,\rho^{\gamma p^*}\mathrm{d}\mu_\sigma)}\right),
\end{aligned}
\end{equation}
where $C = C(n,\kappa,\alpha,p,p^*,R_0) > 0$.

\end{cor}
\begin{proof} We adopt the method presented in \cite{HT2}, which makes use of a duality argument.
Since \( u \in \mathring{W}^{1,p^*}(Q_2, \rho^{\gamma p^*} \mathrm{d}\mu_\sigma) \), it follows from Lemma \ref{WE} that
\begin{equation}\label{eq11}
 u \in L^{p}(Q_2, \rho^{\gamma p} \mathrm{d}\mu_\sigma).
\end{equation}
Let \( \eta \in C_0^\infty(Q_2) \) satisfying \( \eta \equiv 1 \) on \( Q_1 \). 
Similar to \eqref{ctn}, we see that \( w = u\eta \in \mathring{W}^{1,p^*}(Q_2, \rho^{\gamma p^*} \mathrm{d}\mu_\sigma) \) is a weak solution of
\begin{equation}\label{eqw}
 - \partial_i \left( \rho^{-2\alpha} (a^{ij} \partial_j w - \widetilde{F}_i) \right)+\rho^{-2\alpha} \lambda w  = \sqrt{\lambda}\rho^{-2\alpha}\tilde{f} \quad \text{in } Q_2
\end{equation}
with the boundary condition \( w = 0 \) on \( Q_2 \cap \Gamma \) , where
\begin{equation}\label{Ff}
\widetilde{F}_i = F_i \eta - a^{ij} u \partial_j \eta, \quad \tilde{f} = f \eta   - \frac{1}{\sqrt{\lambda}}\partial_i \eta (a^{ij} \partial_j u - F_i),
\end{equation}
and \( \lambda > \lambda_0 R_0^{-2} \) is a constant which will be chosen at the end.

Next, let \( q = p/(p - 1) \), \( q^* = p^*/(p^* - 1) \). Choose \( G = (G_1, \ldots, G_n) \in C_0^\infty(Q_1)^n \) and \( g \in C_0^\infty(Q_1) \) to satisfy
\[
\|G\|_{L^{q^*}(Q_1, \,\rho^{\gamma q^*}\mathrm{d}\mu_\sigma)} = \|g\|_{L^{q^*}(Q_1,\, \rho^{\gamma q^*}\mathrm{d}\mu_\sigma)} = 1.
\]Since $p^*\leq p$, we have $q\leq q^*$. Then, by H\"older's inequality, we obtain
\[
\|G\|_{L^{q}(Q_1, \,\rho^{\gamma q}\mathrm{d}\mu_\sigma)} + \|g\|_{L^{q}(Q_1,\, \rho^{\gamma q}\mathrm{d}\mu_\sigma)} \leq C.
\]
 By Theorem \ref{main theorem}, there is a weak solution \( v \in \mathring{W}^{1,q^*}(Q_2, \rho^{\gamma q^*} \mathrm{d}\mu_\sigma) \) to
\begin{equation}\label{eqvv}
  - \partial_i \left( \rho^{-2\alpha}( a^{ji} \partial_j v - G_i )\right) +\rho^{-2\alpha} \lambda v= \sqrt{\lambda}\rho^{-2\alpha} g \quad \text{in } Q_2
\end{equation}
with the boundary condition \( v = 0 \) on \(  Q_2 \cap \Gamma \). Moreover,
\begin{equation}\label{eq14}
\sqrt{\lambda} \|\rho^\gamma v\|_{L^{q^*}(Q_2, \mathrm{d}\mu_\sigma)} + \|\rho^\gamma D v\|_{L^{q^*}(Q_2, \mathrm{d}\mu_\sigma)} \leq C.
\end{equation}
Since \(q\leq q^* \), we have
\begin{equation}\begin{aligned}\label{eq16}
\|\rho^\gamma Dv\|_{L^{q}(Q_2, \mathrm{d}\mu_\sigma)}
\leq C .
\end{aligned}\end{equation}
Testing \eqref{eqw} and \eqref{eqvv} with \( v \) and \( w \) respectively, we get
\[
\int_{Q_1} \left( (\rho^\gamma D u) \cdot (\rho^\gamma G) + \sqrt{\lambda} (\rho^\gamma u)(\rho^\gamma g) \right) \mathrm{d}\mu_\sigma
= \int_{Q_2} \left( (\rho^\gamma D v) \cdot (\rho^\gamma \widetilde{F}) + \sqrt{\lambda}(\rho^\gamma v) (\rho^\gamma \tilde{f}) \right) \mathrm{d}\mu_\sigma.
\]
Then, it follows from H\"older's inequality that
\begin{equation}\begin{aligned}\label{eq15}
 &\left| \int_{Q_1} \left( (\rho^\gamma D u) \cdot(\rho^{\gamma} G) + \sqrt{\lambda} (\rho^\gamma u)(\rho^{\gamma} g) \right) \mathrm{d}\mu_\sigma \right|
\\ \leq & \|\rho^\gamma D v\|_{L^q(Q_2, \mathrm{d}\mu_\sigma)} \|\rho^\gamma \widetilde{F}\|_{L^p(Q_2, \mathrm{d}\mu_\sigma)} + \sqrt{\lambda}\|\rho^\gamma v\|_{L^{q^*}(Q_2, \mathrm{d}\mu_\sigma)} \|\rho^\gamma \tilde{f}\|_{L^{p^*}(Q_2, \mathrm{d}\mu_\sigma)}.
\end{aligned}\end{equation}
By \eqref{eqvv}, we can see that $v\in \mathring{W}^{1,q^*}(Q_2,\rho^{\gamma p^*}\mathrm{d}\mu_\sigma)$
\begin{equation}\label{}
  - \partial_i \left( \rho^{-2\alpha}( a^{ji} \partial_j v - G_i )\right) =\rho^{-2\alpha} \widetilde{g} \quad \text{in } Q_2
\end{equation} with $\widetilde{g}=\sqrt{\lambda}g-\lambda v$.
It then follows from \eqref{eq15}, \eqref{eq16}, \eqref{Ff}, and the arbitrariness of \( G \) and \( g \) that
\begin{equation}\begin{aligned}\label{eq17}
&\|\rho^\gamma D u\|_{L^p(Q_1, \mathrm{d}\mu_\sigma)} + \sqrt{\lambda} \|\rho^\gamma u\|_{L^p(Q_1, \mathrm{d}\mu_\sigma)}\\
\leq &C \|\rho^\gamma \widetilde{F}\|_{L^p(Q_2, \mathrm{d}\mu_\sigma)} + C \sqrt{\lambda}\|\rho^\gamma \tilde{f}\|_{L^{p^*}(Q_2, \mathrm{d}\mu_\sigma)}\\
\leq & C  \|\rho^\gamma F\|_{L^p(Q_2, \mathrm{d}\mu_\sigma)} + C \|\rho^\gamma u\|_{L^p(Q_2, \mathrm{d}\mu_\sigma)}
+ C \sqrt{\lambda} \|\rho^\gamma f\|_{L^{p^*}(Q_2, \mathrm{d}\mu_\sigma)} \\&+ C  \|\rho^\gamma D u\|_{L^{p^*}(Q_2, \mathrm{d}\mu_\sigma)}\\
\leq & C \left(\|\rho^\gamma D u\|_{L^{p^*}(Q_2, \mathrm{d}\mu_\sigma)}  +\|\rho^\gamma F\|_{L^p(Q_2, \mathrm{d}\mu_\sigma)}+C \sqrt{\lambda} \|\rho^\gamma f\|_{L^{p^*}(Q_2, \mathrm{d}\mu_\sigma)}\right),
\end{aligned}\end{equation} 
where the last inequality follows from Lemma \ref{WE}. Moreover, the constant $C$ is independent of $\lambda$. Finally, we fix  $\lambda$ to be large enough such that $\lambda>\lambda_0R_0^{-2}$, which allows us to conclude the desired result.
\end{proof}

\section*{Declarations}
The authors declare that they have no conflicts of interests.

\end{document}